\documentclass[
final
]{dmtcs-episciences}

\usepackage[utf8]{inputenc}
\usepackage{subfigure}
\usepackage{amsmath}
\usepackage[square]{natbib}

\author{Chuanan Wei\affiliationmark{1}\thanks{I am supported by the National Natural Science Foundations of China (Nos. 11661032, 11301120).}
  \and Xiaoxia Wang\affiliationmark{2}\thanks{And she is supported by the National Natural Science Foundations of China (Nos. 11661032, 11201291).}
  }
\title[Evaluations of series of the $q$-Watson, $q$-Dixon, and $q$-Whipple type]{Evaluations of series of the $q$-Watson, $q$-Dixon, and $q$-Whipple type}
\affiliation{
 Department of Medical Informatics, Hainan Medical University, China\\
  Department of Mathematics, Shanghai University, China}
\received{2016-04-16}
\revised{2017-05-01}
\accepted{2017-06-06}
\begin{document}
\publicationdetails{19}{2017}{1}{19}{1434}
\maketitle
\begin{abstract}
 Using $q$-series identities and series rearrangement, we establish
several extensions of $q$-Watson formulas with two extra integer
parameters. Then they and Sears' transformation formula are utilized
to derive some generalizations of $q$-Dixon formulas and $q$-Whipple
formulas with two extra integer parameters. As special cases of
these results, many interesting evaluations of series of $q$-Watson,
$q$-Dixon, and $q$-Whipple type are displayed.
\end{abstract}

\section{Introduction}
 For two complex numbers $x$ and $q$, define the $q$-shifted
factorial by
\begin{eqnarray*}
 (x;q)_0=1 ~~~\text{and}~~~  (x;q)_n=\prod_{k=0}^{n-1}(1-xq^k)
 ~~~\text{when}~~~ n\in \mathbb{N}.
  \end{eqnarray*}
The fraction form of it reads as
\[\qqdn\qdn\ffnk{ccccc}{q}{a,&b,&\cdots,&c}{\alpha,&\beta,&\cdots,&\gamma}_n
=\frac{(a;q)_n(b;q)_n\cdots(c;q)_n}{(\alpha;q)_n(\beta;q)_n\cdots(\gamma;q)_n}.\]
 Following \cite{gasper}, define the basic hypergeometric series by
\[_{1+r}\phi_s\ffnk{cccc}{q;z}{a_0,a_1,\cdots,a_r}{b_1,b_2,\cdots,b_s}
 =\sum_{k=0}^\infty
\ffnk{ccccc}{q}{a_0,a_1,\cdots,a_r}{b_1,b_2,\cdots,b_s}_k\bigg\{(-1)^kq^{\binom{k}{2}}\bigg\}^{s-r}\qdn\frac{z^k}{(q;q)_k},\]
where $\{a_i\}_{i\geq0}$ and $\{b_j\}_{j\geq1}$ are complex
parameters such that no zero factors appear in the denominators of
the summand on the right hand side. Then Sears' transformation
formula (cf. Equation (2.10.4) of \cite{gasper}) can be expressed as
\begin{eqnarray}\qquad \label{sear}
{_4\phi_3}\ffnk{ccccc}{q;q}{q^{-n},a,b,c}{d,e,q^{1-n}abc/de}
=\ffnk{ccccc}{q}{d/a,de/bc}{d,de/abc}_n
{_4\phi_3}\ffnk{ccccc}{q;q}{q^{-n},a,e/b,e/c}{e,de/bc,q^{1-n}a/d}.
 \end{eqnarray}

There are many interesting formulas for basic hypergeometric series
in the literature. Here we consider a number of $_4\phi_3$
summations. First, the $q$-Watson formula due to \cite{andrews} and
the $q$-Watson formula that is Equation (3.17) of \cite{jain} read,
respectively, as
 \begin{eqnarray}\qquad \label{andrews-watson}
{_4\phi_3}\ffnk{ccccc}{q;q}{q^{-n},q^{1+n}a,\sqrt{c},-\sqrt{c}}
{q\sqrt{a},-q\sqrt{a},c}=c^{n/2}\ffnk{ccccc}{q^2}{q,q^2a/c}{q^{2}a,qc}_{\frac{n}{2}}\chi(n)
 \end{eqnarray}
where $\chi(n)=\{1+(-1)^n\}/2$,
  \begin{eqnarray} \label{jain-watson}
{_4\phi_3}\ffnk{ccccc}{q;q}{a,c,q^{-n},-q^{-n}}
{\sqrt{qac},-\sqrt{qac},q^{-2n}}=\ffnk{ccccc}{q^2}{qa,qc}{q,qac}_n.
\end{eqnarray}
Second, the $q$-Bailey-Dixon formula (cf. page 8 of \cite{gasper})
can be stated as
\begin{eqnarray} \label{Bailey-Dixon}
&&\xxqdn{_4\phi_3}\ffnk{ccccc}{q;q}{q^{-n},a,c,-q^{1-n}/ac}
{q^{1-n}/a,q^{1-n}/c,-ac} =\ffnk{ccccc}{q}{-a,c^2}{-ac,c}_{n}
\ffnk{ccccc}{q^2}{q,q^{n}a^2c^2}{qc^2,q^{n}a^2}_{\frac{n}{2}}\chi(n).
 \end{eqnarray}
By specifying the parameters in \eqref{sear}, we have the relation
 \begin{eqnarray*}
&&\xqdn\qqdn{_4\phi_3}\ffnk{ccccc}{q;q}{a,c,q^{-n},-q^{1+n}a/c}
{qa/c,q^{1+n}a,-q^{-n}c}
=\ffnk{ccccc}{q}{q^{1+n},-qa/c}{q^{1+n}a,-q/c}_{n}
{_4\phi_3}\ffnk{ccccc}{q;q}{a,qa/c^2,q^{-n},-q^{-n}}
{qa/c,-qa/c,q^{-2n}}.
\end{eqnarray*}
The combination of the last formula and \eqref{jain-watson} creates
another $q$-Dixon formula
  \begin{eqnarray} \label{another-Dixon}
{_4\phi_3}\ffnk{ccccc}{q;q}{a,c,q^{-n},-q^{1+n}a/c}
{qa/c,q^{1+n}a,-q^{-n}c}
=\ffnk{ccccc}{q}{q^{1+n},-qa/c}{q^{1+n}a,-q/c}_{n}
\ffnk{ccccc}{q^2}{qa,q^2a/c^2}{q,q^2a^2/c^2}_{n}.
\end{eqnarray}
Third, the $q$-Whipple formula due to \cite{andrews} and the
$q$-Whipple formula that is Equation (3.19) of \cite{jain} read,
respectively, as
 \begin{eqnarray}\qquad \label{andrews-whipple}
{_4\phi_3}\ffnk{ccccc}{q;q}{q^{-n},q^{1+n},\sqrt{qac},-\sqrt{qac}}
{-q,qa,qc}=\frac{(q^{1-n}a;q^2)_n(q^{1-n}c;q^2)_n}{(qa;q)_n(qc;q)_n}\,q^{\binom{n+1}{2}},
 \end{eqnarray}
  \begin{eqnarray} \label{jain-whipple}
{_4\phi_3}\ffnk{ccccc}{q;q}{a,q/a,q^{-n},-q^{-n}}
{c,q^{1-2n}/c,-q}=\frac{(ac;q^2)_n(qc/a;q^2)_n}{(c;q)_{2n}}.
 \end{eqnarray}

By means of contiguous relations for $_3F_2$-series,
\cite{lavoie-b,lavoie-d,lavoie-c} gave a lot of summation formulas
for Watson, Dixon, and Whipple type $_3F_2$-series. For some related
works, the reader may refer to \cite{lavoie-a} and \cite{rathie}. In
2011, \cite{chu-b} established the generalizations of Watson's
$_3F_2$-series identity with two extra integer parameters and
derived several summation formulas for Dixon and Whipple type
$_3F_2$-series according to hypergeometric series identities and
series rearrangement. Let $u$ and $v$ both be integers throughout
the paper. Inspired by Chu's method, we shall explore
 summation formulas for the following six series:
\begin{eqnarray*}
&&\xqdn{_4\phi_3}\ffnk{ccccc}{q;q}{q^{-n},q^{1+u+n}a,\sqrt{c},-q^{v}\!\sqrt{c}}
{q\sqrt{a},-q^{1+u}\!\sqrt{a},q^{v}c},\qquad\quad\:\:
{_4\phi_3}\ffnk{ccccc}{q;q}{q^{u}a,c,q^{-n},-q^{v-n}}
{\sqrt{qac},-q^{u}\!\sqrt{qac},q^{v-2n}},\\
&&\xqdn{_4\phi_3}\ffnk{ccccc}{q;q}{q^{-n},a,c,-q^{1+u+v-n}/ac}
{q^{1+u-n}/a,q^{1+v-n}/c,-ac},\qquad\quad\:\:
{_4\phi_3}\ffnk{ccccc}{q;q}{a,c,q^{-n},-q^{1+u+v+n}a/c}
{q^{1+u}a/c,q^{1+v+n}a,-q^{-n}c},\\
&&\xqdn{_4\phi_3}\ffnk{ccccc}{q;q}{q^{-n},q^{1+u+n},\sqrt{qac},-q^{v}\!\sqrt{qac}}
{-q,q^{1+u+v}a,qc},\qquad
{_4\phi_3}\ffnk{ccccc}{q;q}{a,q^{1+u}/a,q^{-n},-q^{v-n}}
{c,q^{1+u+v-2n}/c,-q},
 \end{eqnarray*}
which can be regarded as terminating $q$-analogues of the formulas
that appear in \cite{lavoie-a}, \cite{chu-b}, and
\cite{lavoie-b,lavoie-d,lavoie-c}. Note that when $u=v=0$ we recover
the series in equations \eqref{andrews-watson}-\eqref{jain-whipple}.

To give just one example of our results, we record our Theorem
\ref{thm-a} as follows:
 \begin{eqnarray*}
&&\xqdn{_4\phi_3}\ffnk{ccccc}{q;q}{q^{-n},q^{1+\ell+n}a,\sqrt{c},-q^{m}\!\sqrt{c}}
{q\sqrt{a},-q^{1+\ell}\!\sqrt{a},q^{m}c}=
 \ffnk{ccccc}{q}{q^{1+\ell+n}a,q^{1+n}\!\sqrt{a}}{q^{1+\ell+2n}a,q\sqrt{a}}_{\ell}
 \\&&\xqdn\:\:=\:
 \sum_{i=0}^{\ell}\sum_{j=0}^{m}q^{2i+(m+n-i)j-\binom{j}{2}}\frac{c^{(n-i)/2}}{a^{i/2}}
 \frac{1-q^{1+2\ell+2n-2i}a}{1-q^{1+2\ell+2n}a}
\ffnk{ccccc}{q}{q^{-\ell},q^{-n},q^{-2\ell-2n-1}/a}
{q,q^{-\ell-2n}/a,q^{-2\ell-n}/a}_i
 \\&&\xqdn\:\:\times\:\ffnk{ccccc}{q}{q^{-m},q^{i-n},q^{1+2\ell+n-i}a,\sqrt{c},-\sqrt{c}}
{q,q^{1+\ell}\!\sqrt{a},-q^{1+\ell}\!\sqrt{a},q^{m}c,q^{j-1}c}_j
\ffnk{ccccc}{q^2}{q,q^{2+2\ell}a/c}{q^{2+2\ell+2j}a,q^{1+2j}c}_{\frac{n-i-j}{2}}\chi(n-i-j),
 \end{eqnarray*}
where $\ell$ and $m$ are both nonnegative integers.

\section{Extensions of $q$-Watson formulas}
\subsection{Extensions of Andrews' $q$-Watson formula}

  In this subsection, we establish several two-parameter extensions of
equation \eqref{andrews-watson}. We begin with the following
one-parameter extension.

\begin{prop}\label{prop-a}
For two complex numbers $\{a,c\}$ and a nonnegative integer $m$,
there holds
 \begin{eqnarray*}
&&\xqdn{_4\phi_3}\ffnk{ccccc}{q;q}{q^{-n},q^{1+n}a,\sqrt{c},-q^{m}\!\sqrt{c}}
{q\sqrt{a},-q\sqrt{a},q^{m}c}\\&&\xqdn\:\:=\:\sum_{j=0}^{m}
q^{(m+n)j-\binom{j}{2}}c^{n/2}
\ffnk{ccccc}{q}{q^{-m},q^{-n},q^{1+n}a,\sqrt{c},-\sqrt{c}}
{q,q\sqrt{a},-q\sqrt{a},q^{m}c,q^{j-1}c}_j
 \\&&\xqdn\:\:\times\:
\ffnk{ccccc}{q^2}{q,q^2a/c}{q^{2+2j}a,q^{1+2j}c}_{\frac{n-j}{2}}\chi(n-j).
 \end{eqnarray*}
\end{prop}

\begin{proof}
 Letting $a\to c/q$, $b\to q^{-k}$, $c\to -\sqrt{c}$ in the
 $_6\phi_5$-series identity (cf. page 42 of \cite{gasper}):
 \begin{eqnarray}\label{terminating-65}
\qquad {_6\phi_5}\ffnk{cccccccc}{q;\frac{q^{1+m}a}{bc}}
{a,\:q\sqrt{a},\:-q\sqrt{a},\:b,\:c,\:q^{-m}}
 {\sqrt{a},\:-\sqrt{a},\:qa/b,\:qa/c,\:q^{1+m}a}
 =\ffnk{ccccc}{q}{qa,qa/bc}{qa/b,qa/c}_m,
 \end{eqnarray}
we get the equation
\begin{eqnarray*}
&&\sum_{j=0}^mq^{mj+\binom{j}{2}}c^{j/2}\frac{1-q^{2j-1}c}{1-q^{j+m-1}c}
\ffnk{ccccc}{q}{q^{-m}}{q}_{j}\ffnk{ccccc}{q}{-\sqrt{c}}{q^{j-1}c}_{m}
\\&&\:\,\times\:
 \frac{(q^{k-j+1};q)_j(q^{k+j}c;q)_{m-j}}{(-q^k\sqrt{c};q)_m}=1.
\end{eqnarray*}
 Then there is the following relation
  \begin{eqnarray*}
&&\xqdn{_4\phi_3}\ffnk{ccccc}{q;q}{q^{-n},q^{1+n}a,\sqrt{c},-q^{m}\!\sqrt{c}}
{q\sqrt{a},-q\sqrt{a},q^{m}c}=\sum_{k=0}^n\ffnk{ccccc}{q}{q^{-n},q^{1+n}a,\sqrt{c},-q^{m}\!\sqrt{c}}
{q,q\sqrt{a},-q\sqrt{a},q^{m}c}_kq^k\\
&&\xqdn\:\:=\:\sum_{k=0}^n\ffnk{ccccc}{q}{q^{-n},q^{1+n}a,\sqrt{c},-q^{m}\!\sqrt{c}}
{q,q\sqrt{a},-q\sqrt{a},q^{m}c}_kq^k
\sum_{j=0}^mq^{mj+\binom{j}{2}}c^{j/2}\frac{1-q^{2j-1}c}{1-q^{j+m-1}c}
\\&&\xqdn\:\:\times\:\ffnk{ccccc}{q}{q^{-m}}{q}_{j}\ffnk{ccccc}{q}{-\sqrt{c}}{q^{j-1}c}_{m}
 \frac{(q^{k-j+1};q)_j(q^{k+j}c;q)_{m-j}}{(-q^k\sqrt{c};q)_m}.
 \end{eqnarray*}
Interchange the summation order for the last double sum to obtain
  \begin{eqnarray*}
&&\xqdn{_4\phi_3}\ffnk{ccccc}{q;q}{q^{-n},q^{1+n}a,\sqrt{c},-q^{m}\!\sqrt{c}}
{q\sqrt{a},-q\sqrt{a},q^{m}c}=
\sum_{j=0}^mq^{mj+\binom{j}{2}}c^{j/2}\frac{1-q^{2j-1}c}{1-q^{j+m-1}c}
\\&&\xqdn\:\:\times\:\ffnk{ccccc}{q}{q^{-m}}{q}_{j}\ffnk{ccccc}{q}{-\sqrt{c}}{q^{j-1}c}_{m}
\\&&\xqdn\:\:\times\: \sum_{k=j}^n\ffnk{ccccc}{q}{q^{-n},q^{1+n}a,\sqrt{c},-q^{m}\!\sqrt{c}}
{q,q\sqrt{a},-q\sqrt{a},q^{m}c}_kq^k
 \frac{(q^{k-j+1};q)_j(q^{k+j}c;q)_{m-j}}{(-q^k\!\sqrt{c};q)_m}.
 \end{eqnarray*}
Shifting the index $k\to i+j$ for the sum on the last line, the
result reads as
 \begin{eqnarray}
&&{_4\phi_3}\ffnk{ccccc}{q;q}{q^{-n},q^{1+n}a,\sqrt{c},-q^{m}\!\sqrt{c}}
{q\sqrt{a},-q\sqrt{a},q^{m}c}\nonumber\\&&\:\:=\!\: \sum_{j=0}^m
 q^{mj+\binom{j+1}{2}}c^{j/2}\ffnk{ccccc}{q}{q^{-m},q^{-n},q^{1+n}a,\sqrt{c},-\sqrt{c}}
{q,q\sqrt{a},-q\sqrt{a},q^{m}c,q^{j-1}c}_j
\nonumber\\&&\:\:\times\,\:{_4\phi_3}\ffnk{ccccc}{q;q}{q^{j-n},q^{1+n+j}a,q^j\!\sqrt{c},-q^j\!\sqrt{c}}
{q^{1+j}\!\sqrt{a},-q^{1+j}\!\sqrt{a},q^{2j}c}. \label{equation-a}
 \end{eqnarray}
 Calculating the $_4\phi_3$-series on the last line by
\eqref{andrews-watson}, we complete the proof Proposition
\ref{prop-a}.
\end{proof}

\begin{exam}[$m=1$ in Proposition \ref{prop-a}]
 \begin{eqnarray*}
 &&{_4\phi_3}\ffnk{ccccc}{q;q}{q^{-n},q^{1+n}a,\sqrt{c},-q\sqrt{c}}
{q\sqrt{a},-q\sqrt{a},qc}\\
&&\:=\:
\begin{cases}
c^s\ffnk{ccccc}{q^2}{q,q^2a/c}{q^{2}a,qc}_s,&\qqdn n=2s;
\\[4mm]
c^{\frac{1}{2}+s}\!\ffnk{ccccc}{q^2}{q}{qc}_{1+s}\!\!
 \ffnk{ccccc}{q^2}{q^2a/c}{q^2a}_{s}\!,&\qqdn n=1+2s.
\end{cases}
\end{eqnarray*}
\end{exam}

\begin{thm}\label{thm-a}
For two complex numbers $\{a,c\}$ and two nonnegative integers
$\{\ell,m\}$, there holds
 \begin{eqnarray*}
&&\xqdn{_4\phi_3}\ffnk{ccccc}{q;q}{q^{-n},q^{1+\ell+n}a,\sqrt{c},-q^{m}\!\sqrt{c}}
{q\sqrt{a},-q^{1+\ell}\!\sqrt{a},q^{m}c}=
 \ffnk{ccccc}{q}{q^{1+\ell+n}a,q^{1+n}\!\sqrt{a}}{q^{1+\ell+2n}a,q\sqrt{a}}_{\ell}
 \\&&\xqdn\:\:=\:
 \sum_{i=0}^{\ell}\sum_{j=0}^{m}q^{2i+(m+n-i)j-\binom{j}{2}}\frac{c^{(n-i)/2}}{a^{i/2}}
 \frac{1-q^{1+2\ell+2n-2i}a}{1-q^{1+2\ell+2n}a}
\ffnk{ccccc}{q}{q^{-\ell},q^{-n},q^{-2\ell-2n-1}/a}
{q,q^{-\ell-2n}/a,q^{-2\ell-n}/a}_i
 \\&&\xqdn\:\:\times\:\ffnk{ccccc}{q}{q^{-m},q^{i-n},q^{1+2\ell+n-i}a,\sqrt{c},-\sqrt{c}}
{q,q^{1+\ell}\!\sqrt{a},-q^{1+\ell}\!\sqrt{a},q^{m}c,q^{j-1}c}_j
\ffnk{ccccc}{q^2}{q,q^{2+2\ell}a/c}{q^{2+2\ell+2j}a,q^{1+2j}c}_{\frac{n-i-j}{2}}\chi(n-i-j).
 \end{eqnarray*}
\end{thm}

\begin{proof}
 Performing the replacements $m\to \ell$, $a\to q^{-2\ell-2n-1}/a$, $b\to q^{k-n}$, $c\to q^{-\ell-n}/\sqrt{a}$
 in \eqref{terminating-65}, we get the equation
 \begin{eqnarray*}
&&\ffnk{ccccc}{q}{q^{1+\ell+n}a,q^{1+n}\!\sqrt{a}}{q^{1+\ell+2n}a,q^{1+k}\!\sqrt{a}}_{\ell}
 \sum_{i=0}^{\ell}\frac{q^{2i}}{a^{i/2}}\frac{1-q^{1+2\ell+2n-2i}a}{1-q^{1+2\ell+2n}a}
\ffnk{ccccc}{q}{q^{-\ell},q^{k-n},q^{-2\ell-2n-1}/a}{q,q^{-\ell-2n}/a,q^{-2\ell-n}/a}_i
\\&&\:\,\times\:\ffnk{ccccc}{q}{q^{\ell+n+k+1}a}{q^{\ell+n+1}a}_{\ell-i}=1.
 \end{eqnarray*}
 Then there exists the following relation
\begin{eqnarray*}
&&\xqdn{_4\phi_3}\ffnk{ccccc}{q;q}{q^{-n},q^{1+\ell+n}a,\sqrt{c},-q^{v}\!\sqrt{c}}
{q\sqrt{a},-q^{1+\ell}\!\sqrt{a},q^{v}c}
 =\sum_{k=0}^n\ffnk{ccccc}{q}{q^{-n},q^{1+\ell+n}a,\sqrt{c},-q^{v}\!\sqrt{c}}
{q,q\sqrt{a},-q^{1+\ell}\!\sqrt{a},q^{v}c}_kq^k\\
&&\xqdn\:\:=\:\sum_{k=0}^n\ffnk{ccccc}{q}{q^{-n},q^{1+\ell+n}a,\sqrt{c},-q^{v}\!\sqrt{c}}
{q,q\sqrt{a},-q^{1+\ell}\!\sqrt{a},q^{v}c}_kq^k
\ffnk{ccccc}{q}{q^{1+\ell+n}a,q^{1+n}\!\sqrt{a}}{q^{1+\ell+2n}a,q^{1+k}\!\sqrt{a}}_{\ell}
\\&&\xqdn\:\:\times\:\sum_{i=0}^{\ell}\frac{q^{2i}}{a^{i/2}}\frac{1-q^{1+2\ell+2n-2i}a}{1-q^{1+2\ell+2n}a}
\ffnk{ccccc}{q}{q^{-\ell},q^{k-n},q^{-2\ell-2n-1}/a}{q,q^{-\ell-2n}/a,q^{-2\ell-n}/a}_i
\ffnk{ccccc}{q}{q^{\ell+n+k+1}a}{q^{\ell+n+1}a}_{\ell-i}\\
\\&&\xqdn\:\:=\:\ffnk{ccccc}{q}{q^{1+\ell+n}a,q^{1+n}\!\sqrt{a}}{q^{1+\ell+2n}a,q\sqrt{a}}_{\ell}
\sum_{i=0}^{\ell}\frac{q^{2i}}{a^{i/2}}\frac{1-q^{1+2\ell+2n-2i}a}{1-q^{1+2\ell+2n}a}
\ffnk{ccccc}{q}{q^{-\ell},q^{-2\ell-2n-1}/a}{q,q^{-\ell-2n}/a,q^{-2\ell-n}/a}_i
\\&&\xqdn\:\:\times\:\sum_{k=0}^n\ffnk{ccccc}{q}{q^{-n},q^{1+\ell+n}a,\sqrt{c},-q^{v}\!\sqrt{c}}
{q,q\sqrt{a},-q^{1+\ell}\!\sqrt{a},q^{v}c}_kq^k\frac{(q\sqrt{a};q)_{\ell}(q^{k-n};q)_i}{(q^{1+k}\!\sqrt{a};q)_{\ell}}
 \ffnk{ccccc}{q}{q^{\ell+n+k+1}a}{q^{\ell+n+1}a}_{\ell-i}.
 \end{eqnarray*}
After some routine simplification, the result reads as
  \begin{eqnarray}
&&\xqdn{_4\phi_3}\ffnk{ccccc}{q;q}{q^{-n},q^{1+\ell+n}a,\sqrt{c},-q^{v}\!\sqrt{c}}
{q\sqrt{a},-q^{1+\ell}\!\sqrt{a},q^{v}c}
=\ffnk{ccccc}{q}{q^{1+\ell+n}a,q^{1+n}\!\sqrt{a}}{q^{1+\ell+2n}a,q\sqrt{a}}_{\ell}
 \nonumber\\\nonumber&&\xqdn\:\:\times\:
 \sum_{i=0}^{\ell}\frac{q^{2i}}{a^{i/2}}\frac{1-q^{1+2\ell+2n-2i}a}{1-q^{1+2\ell+2n}a}
\ffnk{ccccc}{q}{q^{-\ell},q^{-n},q^{-2\ell-2n-1}/a}{q,q^{-\ell-2n}/a,q^{-2\ell-n}/a}_i
\\&&\xqdn\:\:\times\:{_4\phi_3}\ffnk{ccccc}{q;q}{q^{i-n},q^{1+2\ell+n-i}a,\sqrt{c},-q^{v}\!\sqrt{c}}
{q^{1+\ell}\!\sqrt{a},-q^{1+\ell}\!\sqrt{a},q^{v}c}.
\label{equation-c}
 \end{eqnarray}
Setting $v=m$ in \eqref{equation-c} and evaluating the
$_4\phi_3$-series on the right hand side by Proposition
\ref{prop-a}, we finish the proof of Theorem \ref{thm-a}.
\end{proof}

\begin{exam}[$\ell=1,m=0$ in Theorem \ref{thm-a}]
\begin{eqnarray*}
 &&\xqdn\xxqdn{_4\phi_3}\ffnk{ccccc}{q;q}{q^{-n},q^{2+n}a,\sqrt{c},-\sqrt{c}}
{q\sqrt{a},-q^2\!\sqrt{a},c}\\
&&\xqdn\xxqdn\:=\:
\begin{cases}
\frac{(1+q\sqrt{a})c^s}{1+q^{1+2s}\sqrt{a}}\ffnk{ccccc}{q^2}{q,q^4a/c}{q^{2}a,qc}_s,&\qqdn
n=2s;
\\[4mm]
\frac{(q^2-q)\sqrt{a}\,c^s}{(1-q\sqrt{a})(1+q^{2+2s}\!\sqrt{a})}\ffnk{ccccc}{q^2}{q^3,q^4a/c}{q^{4}a,qc}_s\!,&\qqdn
n=1+2s.
\end{cases}
 \end{eqnarray*}
\end{exam}

\begin{exam}[$\ell=1,m=1$ in Theorem \ref{thm-a}]
 \begin{eqnarray*}
 &&{_4\phi_3}\ffnk{ccccc}{q;q}{q^{-n},q^{2+n}a,\sqrt{c},-q\sqrt{c}}
{q\sqrt{a},-q^2\!\sqrt{a},qc}\\
&&\:=\:
\begin{cases}
\frac{(1+q\sqrt{a})(\sqrt{c}+q^{1+2s}\!\sqrt{a})c^s}{(\sqrt{c}+q\sqrt{a})(1+q^{1+2s}\!\sqrt{a})}\ffnk{ccccc}{q^2}{q,q^2a/c}{q^{2}a,qc}_s,&\qqdn
n=2s;
\\[4mm]
\frac{(1-q)(\sqrt{c}-q\sqrt{a})(1+q^{2+2s}\!\sqrt{ac})c^s}{(1-q\sqrt{a})(1-qc)(1+q^{2+2s}\!\sqrt{a})}\ffnk{ccccc}{q^2}{q^3,q^4a/c}{q^{4}a,q^3c}_s\!,&\qqdn
n=1+2s.
\end{cases}
 \end{eqnarray*}
\end{exam}

Replacing $\sqrt{a}$ by $-q^{-\ell}\!\sqrt{a}$ in Theorem
\ref{thm-a}, we obtain the equation
  \begin{eqnarray}
&&\xqdn{_4\phi_3}\ffnk{ccccc}{q;q}{q^{-n},q^{1-\ell+n}a,\sqrt{c},-q^{m}\!\sqrt{c}}
{q\sqrt{a},-q^{1-\ell}\!\sqrt{a},q^{m}c}=
 \ffnk{ccccc}{q}{q^{-n}/a,-q^{-n}/\sqrt{a}}{q^{-2n}/a,-1/\sqrt{a}}_{\ell}
 \nonumber\\\nonumber&&\xqdn\:\:=\:
 \sum_{i=0}^{\ell}\sum_{j=0}^{m}(-1)^iq^{(\ell+2)i+(m+n-i)j-\binom{j}{2}}\frac{c^{(n-i)/2}}{a^{i/2}}
 \frac{1-q^{1+2n-2i}a}{1-q^{1+2n}a}
\ffnk{ccccc}{q}{q^{-\ell},q^{-n},q^{-2n-1}/a}
{q,q^{-n}/a,q^{\ell-2n}/a}_i
 \nonumber\\\label{equivalence-a}
 &&\xqdn\:\:\times\:\ffnk{ccccc}{q}{q^{-m},q^{i-n},q^{1+n-i}a,\sqrt{c},-\sqrt{c}}
{q,q\sqrt{a},-q\sqrt{a},q^{m}c,q^{j-1}c}_j
\ffnk{ccccc}{q^2}{q,q^2a/c}{q^{2+2j}a,q^{1+2j}c}_{\frac{n-i-j}{2}}\chi(n-i-j).
  \end{eqnarray}
Employing the substitution $\sqrt{c}\to-q^{-m}\!\sqrt{c}$ in Theorem
\ref{thm-a}, we get the formula
  \begin{eqnarray}
&&\xqdn{_4\phi_3}\ffnk{ccccc}{q;q}{q^{-n},q^{1+\ell+n}a,\sqrt{c},-q^{-m}\!\sqrt{c}}
{q\sqrt{a},-q^{1+\ell}\!\sqrt{a},q^{-m}c}=
 \ffnk{ccccc}{q}{q^{1+\ell+n}a,q^{1+n}\!\sqrt{a}}{q^{1+\ell+2n}a,q\sqrt{a}}_{\ell}
 \nonumber\\\nonumber&&\xqdn\:\:=\:
 \sum_{i=0}^{\ell}\sum_{j=0}^{m}(-1)^jq^{2i+(n-i-j)(j-m)+\binom{j+1}{2}}\frac{c^{(n-i)/2}}{a^{i/2}}
 \frac{1-q^{1+2\ell+2n-2i}a}{1-q^{1+2\ell+2n}a}
 \nonumber\\\nonumber&&\xqdn\:\:\times\:\ffnk{ccccc}{q}{q^{-\ell},q^{-n},q^{-2\ell-2n-1}/a}
{q,q^{-\ell-2n}/a,q^{-2\ell-n}/a}_i
\ffnk{ccccc}{q}{q^{-m},q^{i-n},q^{1+2\ell+n-i}a,q^{-m}\!\sqrt{c},-q^{-m}\!\sqrt{c}}
{q,q^{1+\ell}\!\sqrt{a},-q^{1+\ell}\!\sqrt{a},q^{-m}c,q^{j-2m-1}c}_j
\\\label{equivalence-b}&&\xqdn\:\:\times\:
\ffnk{ccccc}{q^2}{q,q^{2+2\ell+2m}a/c}{q^{2+2\ell+2j}a,q^{1-2m+2j}c}_{\frac{n-i-j}{2}}\chi(n-i-j).
\end{eqnarray}
Letting
$\sqrt{a}\to-q^{-\ell}\!\sqrt{a},\sqrt{c}\to-q^{-m}\!\sqrt{c}$ in
Theorem \ref{thm-a}, we obtain the result
 \begin{eqnarray}
&&\xqdn{_4\phi_3}\ffnk{ccccc}{q;q}{q^{-n},q^{1-\ell+n}a,\sqrt{c},-q^{-m}\!\sqrt{c}}
{q\sqrt{a},-q^{1-\ell}\!\sqrt{a},q^{-m}c}=
 \ffnk{ccccc}{q}{q^{-n}/a,-q^{-n}/\sqrt{a}}{q^{-2n}/a,-1/\sqrt{a}}_{\ell}
 \nonumber\\\nonumber&&\xqdn\:\:=\:
 \sum_{i=0}^{\ell}\sum_{j=0}^{m}(-1)^{i+j}q^{(\ell+2)i+(n-i-j)(j-m)+\binom{j+1}{2}}\frac{c^{(n-i)/2}}{a^{i/2}}
 \frac{1-q^{1+2n-2i}a}{1-q^{1+2n}a}
 \nonumber\\\nonumber&&\xqdn\:\:\times\:
\ffnk{ccccc}{q}{q^{-\ell},q^{-n},q^{-2n-1}/a}{q,q^{-n}/a,q^{\ell-2n}/a}_i
 \ffnk{ccccc}{q}{q^{-m},q^{i-n},q^{1+n-i}a,q^{-m}\!\sqrt{c},-q^{-m}\!\sqrt{c}}
{q,q\sqrt{a},-q\sqrt{a},q^{-m}c,q^{j-2m-1}c}_j
 \\\label{equivalence-c}&&\xqdn\:\:\times\:
\ffnk{ccccc}{q^2}{q,q^{2+2m}a/c}{q^{2+2j}a,q^{1-2m+2j}c}_{\frac{n-i-j}{2}}\chi(n-i-j).
\end{eqnarray}

\textbf{Remark}: Theorem \ref{thm-a}, \eqref{equivalence-a},
\eqref{equivalence-b} and \eqref{equivalence-c} are equivalent to
each other, but the corresponding hypergeometric series identities
 are essentially different.

\subsection{Extensions of Jain's $q$-Watson formula}
In this subsection, we prove some two-parameter extensions of
equation \eqref{jain-watson}. We start with the following
one-parameter extension.

\begin{prop}\label{prop-b}
For two complex numbers $\{a,c\}$ and a nonnegative integer $m$ with
$m\leq n$, there holds
  \begin{eqnarray*}
&&{_4\phi_3}\ffnk{ccccc}{q;q}{a,c,q^{-n},-q^{m-n}}
{\sqrt{qac},-\sqrt{qac},q^{m-2n}}\\&&\:\:=\:\sum_{j=0}^{m}
q^{(m-n)j+\binom{j+1}{2}} \ffnk{ccccc}{q}{q^{-m},q^{-n},-q^{-n},a,c}
{q,\sqrt{qac},-\sqrt{qac},q^{m-2n},q^{j-2n-1}}_j
        \nonumber  \\&&\:\:\times\:
\ffnk{ccccc}{q^2}{q^{1+j}a,q^{1+j}c}{q,q^{1+2j}ac}_{n-j}.
  \end{eqnarray*}
\end{prop}

\begin{proof}
Performing the replacement $a\to q^{-1-n}a$ in \eqref{equation-a},
we have
\begin{eqnarray}
&&{_4\phi_3}\ffnk{ccccc}{q;q}{q^{-n},a,\sqrt{c},-q^{m}\!\sqrt{c}}
{\sqrt{q^{1-n}a},-\sqrt{q^{1-n}a},q^{m}c}
 \nonumber\\\nonumber&&\:\:=\!\:
\sum_{j=0}^mq^{mj+\binom{j+1}{2}}c^{j/2}\ffnk{ccccc}{q}{q^{-m},q^{-n},a,\sqrt{c},-\sqrt{c}}
{q,\sqrt{q^{1-n}a},-\sqrt{q^{1-n}a},q^{m}c,q^{j-1}c}_j
\\\label{equivalence-d}&&\:\:\times\,\:
 {_4\phi_3}\ffnk{ccccc}{q;q}{q^{j-n},q^{j}a,q^j\!\sqrt{c},-q^j\!\sqrt{c}}
{q^{j}\!\sqrt{q^{1-n}a},-q^{j}\!\sqrt{q^{1-n}a},q^{2j}c}.
\end{eqnarray}
Replacing $c$ and $q^{-n}$ by $q^{-2n}$ and $c$, respectively, in
\eqref{equivalence-d}, the result reads as
\begin{eqnarray*}
 &&{_4\phi_3}\ffnk{ccccc}{q;q}{a,c,q^{-n},-q^{m-n}}
{\sqrt{qac},-\sqrt{qac},q^{m-2n}}\\&&\:\:=\:\sum_{j=0}^{m}q^{(m-n)j+\binom{j+1}{2}}
\ffnk{ccccc}{q}{q^{-m},q^{-n},-q^{-n},a,c}
{q,\sqrt{qac},-\sqrt{qac},q^{m-2n},q^{j-2n-1}}_j
        \nonumber  \\&&\:\:\times\:\,
{_4\phi_3}\ffnk{ccccc}{q;q}{q^{j}a,q^{j}c,q^{j-n},-q^{j-n}}
{q^{j}\!\sqrt{qac},-q^{j}\!\sqrt{qac},q^{2j-2n}}.
\end{eqnarray*}
Calculating the $_4\phi_3$-series on the last line by
\eqref{jain-watson}, we establish the proposition.
\end{proof}

\begin{exam}[$m=1$ in Proposition \ref{prop-b}: $n\geq1$]
\begin{eqnarray*}
\qquad\qquad\:\:\:
{_4\phi_3}\ffnk{ccccc}{q;q}{a,c,q^{-n},-q^{1-n}}{\sqrt{qac},-\sqrt{qac},q^{1-2n}}=
\ffnk{ccccc}{q^2}{qa,qc}{q,qac}_{n}+q^n\ffnk{ccccc}{q^2}{a,c}{q,qac}_{n}.
\end{eqnarray*}
\end{exam}

\begin{prop}\label{prop-c}
For two complex numbers $\{a,c\}$ and a nonnegative integer $m$,
there holds
 \begin{eqnarray*}
&&{_4\phi_3}\ffnk{ccccc}{q;q}{a,c,q^{-n},-q^{-m-n}}
{\sqrt{qac},-\sqrt{qac},q^{-m-2n}}\\&&\:\:=\:\sum_{j=0}^{m}(-1)^j
q^{\binom{j+1}{2}-nj} \ffnk{ccccc}{q}{q^{-m},q^{-m-n},-q^{-m-n},a,c}
{q,\sqrt{qac},-\sqrt{qac},q^{-m-2n},q^{j-2m-2n-1}}_j
        \nonumber  \\&&\:\:\times\:
\ffnk{ccccc}{q^2}{q^{1+j}a,q^{1+j}c}{q,q^{1+2j}ac}_{m+n-j}.
 \end{eqnarray*}
\end{prop}

\begin{proof}
Employing the substitution $\sqrt{c}\to-q^{-m}\!\sqrt{c}$ in
\eqref{equivalence-d}, we get
  \begin{eqnarray*}
&&{_4\phi_3}\ffnk{ccccc}{q;q}{q^{-n},a,\sqrt{c},-q^{-m}\!\sqrt{c}}
{\sqrt{q^{1-n}a},-\sqrt{q^{1-n}a},q^{-m}c}\nonumber\\&&\:\:=\!\:
\sum_{j=0}^m(-1)^j
 q^{\binom{j+1}{2}}c^{j/2}\ffnk{ccccc}{q}{q^{-m},q^{-n},a,q^{-m}\!\sqrt{c},-q^{-m}\!\sqrt{c}}
{q,\sqrt{q^{1-n}a},-\sqrt{q^{1-n}a},q^{-m}c,q^{j-2m-1}c}_j
\nonumber\\&&\:\:\times\,\:{_4\phi_3}\ffnk{ccccc}{q;q}{q^{j-n},q^{j}a,q^{j-m}\!\sqrt{c},-q^{j-m}\!\sqrt{c}}
{q^{j}\!\sqrt{q^{1-n}a},-q^{j}\!\sqrt{q^{1-n}a},q^{2j-2m}c}.
\end{eqnarray*}
Letting $c\to q^{-2n}$, $q^{-n}\to c$ in the last relation, the
result reads as
\begin{eqnarray*}
 &&{_4\phi_3}\ffnk{ccccc}{q;q}{a,c,q^{-n},-q^{-m-n}}
{\sqrt{qac},-\sqrt{qac},q^{-m-2n}}\\&&\:\:=\:\sum_{j=0}^{m}(-1)^jq^{\binom{j+1}{2}-nj}
\ffnk{ccccc}{q}{q^{-m},q^{-m-n},-q^{-m-n},a,c}
{q,\sqrt{qac},-\sqrt{qac},q^{-m-2n},q^{j-2m-2n-1}}_j
        \nonumber  \\&&\:\:\times\:\,
{_4\phi_3}\ffnk{ccccc}{q;q}{q^{j}a,q^{j}c,q^{j-m-n},-q^{j-m-n}}
{q^{j}\!\sqrt{qac},-q^{j}\!\sqrt{qac},q^{2j-2m-2n}}.
\end{eqnarray*}
Evaluating the $_4\phi_3$-series on the last line by
\eqref{jain-watson}, we establish the proposition.
\end{proof}

\begin{exam}[$m=1$ in Proposition \ref{prop-c}]
 \begin{eqnarray*}
&&{_4\phi_3}\ffnk{ccccc}{q;q}{a,c,q^{-n},-q^{-1-n}}{\sqrt{qac},-\sqrt{qac},q^{-1-2n}}
\\&&\:\:=\:
\ffnk{ccccc}{q^2}{qa,qc}{q,qac}_{n+1}-q^{n+1}\ffnk{ccccc}{q^2}{a,c}{q,qac}_{n+1}.
\end{eqnarray*}
\end{exam}

\begin{thm}\label{thm-b}
For two complex numbers $\{a,c\}$ and two nonnegative integers
$\{\ell, m\}$ with $m\leq n$, there holds
  \begin{eqnarray*}
&&{_4\phi_3}\ffnk{ccccc}{q;q}{q^{\ell}a,c,q^{-n},-q^{m-n}}
{\sqrt{qac},-q^{\ell}\!\sqrt{qac},q^{m-2n}}
=\ffnk{ccccc}{q}{q^{\ell}a,\sqrt{qa/c}}{q^{\ell}a/c,\sqrt{qac}}_{\ell}
\\&&\:\:\times\:
\sum_{i=0}^{\ell}\sum_{j=0}^{m}q^{\frac{5i}{2}+(m-n)j+\binom{j+1}{2}}\frac{1}{(ac)^{i/2}}
 \frac{1-q^{2\ell-2i}a/c}{1-q^{2\ell}a/c}
\ffnk{ccccc}{q}{q^{-\ell},c,q^{-2\ell}c/a}{q,q^{1-2\ell}/a,q^{1-\ell}c/a}_i
 \\&&\:\:\times\:\ffnk{ccccc}{q}{q^{-m},q^{-n},-q^{-n},q^{2\ell-i}a,q^ic}
{q,q^{\ell}\!\sqrt{qac},-q^{\ell}\!\sqrt{qac},q^{m-2n},q^{j-2n-1}}_j
\ffnk{ccccc}{q^2}{q^{1+2\ell-i+j}a,q^{1+i+j}c}{q,q^{1+2\ell+2j}ac}_{n-j}.
 \end{eqnarray*}
\end{thm}

\begin{proof}
Performing the replacement $a\to q^{-n-1}a$ in \eqref{equation-c},
we obtain
  \begin{eqnarray*}
&&\xqdn{_4\phi_3}\ffnk{ccccc}{q;q}{q^{-n},q^{\ell}a,\sqrt{c},-q^{v}\!\sqrt{c}}
{\sqrt{q^{1-n}a},-q^{\ell}\!\sqrt{q^{1-n}a},q^{v}c}
=\ffnk{ccccc}{q}{q^{\ell}a,\sqrt{q^{1+n}a}}{q^{\ell+n}a,\sqrt{q^{1-n}a}}_{\ell}
 \\&&\xqdn\:\:\times\:
 \sum_{i=0}^{\ell}\frac{q^{(5+n)i/2}}{a^{i/2}}\frac{1-q^{2\ell+n-2i}a}{1-q^{2\ell+n}a}
\ffnk{ccccc}{q}{q^{-\ell},q^{-n},q^{-2\ell-n}/a}{q,q^{1-\ell-n}/a,q^{1-2\ell}/a}_i
\\&&\xqdn\:\:\times\:{_4\phi_3}\ffnk{ccccc}{q;q}{q^{i-n},q^{2\ell-i}a,\sqrt{c},-q^{v}\!\sqrt{c}}
{q^{\ell}\!\sqrt{q^{1-n}a},-q^{\ell}\!\sqrt{q^{1-n}a},q^{v}c}.
\end{eqnarray*}
Replacing $q^{-n}$ and $c$ by $c$ and $q^{-2n}$, respectively, in
the last relation, the result reads as
\begin{eqnarray}
&&\xqdn{_4\phi_3}\ffnk{ccccc}{q;q}{q^{\ell}a,c,q^{-n},-q^{v-n}}
{\sqrt{qac},-q^{\ell}\!\sqrt{qac},q^{v-2n}}
=\ffnk{ccccc}{q}{q^{\ell}a,\sqrt{qa/c}}{q^{\ell}a/c,\sqrt{qac}}_{\ell}
 \nonumber\\\nonumber&&\xqdn\:\:\times\:
 \sum_{i=0}^{\ell}\frac{q^{5i/2}}{(ac)^{i/2}}\frac{1-q^{2\ell-2i}a/c}{1-q^{2\ell}a/c}
\ffnk{ccccc}{q}{q^{-\ell},c,q^{-2\ell}c/a}{q,q^{1-2\ell}/a,q^{1-\ell}c/a}_i
\\&&\xqdn\:\:\times\:{_4\phi_3}\ffnk{ccccc}{q;q}{q^{2\ell-i}a,q^ic,q^{-n},-q^{v-n}}
{q^{\ell}\!\sqrt{qac},-q^{\ell}\!\sqrt{qac},q^{v-2n}}.
\label{equation-aa}
\end{eqnarray}
Taking $v=m$ in \eqref{equation-aa} and calculating the
$_4\phi_3$-series on the right hand side by Proposition
\ref{prop-b}, we establish the theorem.
\end{proof}

\begin{exam}[$\ell=1,m=0$ in Theorem \ref{thm-b}]
\begin{eqnarray*}
&&\xxqdn\qqdn\:{_4\phi_3}\ffnk{ccccc}{q;q}{qa,c,q^{-n},-q^{-n}}{\sqrt{qac},-\sqrt{q^3ac},q^{-2n}}\\
&&\xxqdn\qqdn\:\:\:=\:\frac{1-qa}{(1-\sqrt{qac})(1+\sqrt{qa/c})}\ffnk{ccccc}{q^2}{q^3a,qc}{q,q^3ac}_{n}\\
&&\xxqdn\qqdn\:\:\:+\:\,\frac{1-c}{(1-\sqrt{qac})(1+\sqrt{c/qa})}\ffnk{ccccc}{q^2}{q^2a,q^2c}{q,q^3ac}_{n}.
\end{eqnarray*}
\end{exam}

\begin{exam}[$\ell=1,m=1$ in Theorem \ref{thm-b}: $n\geq1$]
\begin{eqnarray*}
&&{_4\phi_3}\ffnk{ccccc}{q;q}{qa,c,q^{-n},-q^{1-n}}{\sqrt{qac},-\sqrt{q^3ac},q^{1-2n}}\\
&&\:\:=\:\frac{(1-\sqrt{q^{1+2n}ac})(1+\sqrt{q^{1+2n}a/c})}{(1-\sqrt{qac})(1+\sqrt{qa/c})}\ffnk{ccccc}{q^2}{qa,qc}{q,q^3ac}_{n}\\
&&\:\:+\:\,\frac{(1-\sqrt{q^{1+2n}ac})(q^n+\sqrt{qa/c})}{(1-\sqrt{qac})(1+\sqrt{qa/c})}\ffnk{ccccc}{q^2}{q^2a,c}{q,q^3ac}_{n}.
\end{eqnarray*}
\end{exam}

Employing the substitution $\sqrt{a}\to-q^{-\ell}\!\sqrt{a}$ in
Theorem \ref{thm-b}, we obtain the equation
\begin{eqnarray}
&&\xxqdn{_4\phi_3}\ffnk{ccccc}{q;q}{q^{-\ell}a,c,q^{-n},-q^{m-n}}
{\sqrt{qac},-q^{-\ell}\!\sqrt{qac},q^{m-2n}}
=\ffnk{ccccc}{q}{q/a,-\sqrt{qc/a}}{qc/a,-\sqrt{q/ac}}_{\ell}
\nonumber\\\nonumber&&\xxqdn\:\:\times\:
\sum_{i=0}^{\ell}\sum_{j=0}^{m}(-1)^iq^{\frac{(1+2\ell)i}{2}+(m-n)j+\binom{j+1}{2}}\frac{1}{(ac)^{i/2}}
 \frac{1-q^{2i}c/a}{1-c/a}
\ffnk{ccccc}{q}{q^{-\ell},c,c/a}{q,q/a,q^{1+\ell}c/a}_i
 \nonumber\\\label{equivalence-e}
 &&\xxqdn\:\:\times\:\ffnk{ccccc}{q}{q^{-m},q^{-n},-q^{-n},q^{-i}a,q^ic}
{q,\sqrt{qac},-\sqrt{qac},q^{m-2n},q^{j-2n-1}}_j
\ffnk{ccccc}{q^2}{q^{1-i+j}a,q^{1+i+j}c}{q,q^{1+2j}ac}_{n-j}.
\end{eqnarray}

Setting $v=-m$ in \eqref{equation-aa} and evaluating the
$_4\phi_3$-series on the right hand side by Proposition
\ref{prop-c}, we get the following theorem.

\begin{thm}\label{thm-c}
For two complex numbers $\{a,c\}$ and two nonnegative integers
$\{\ell,m\}$, there holds
 \begin{eqnarray*}
&&{_4\phi_3}\ffnk{ccccc}{q;q}{q^{\ell}a,c,q^{-n},-q^{-m-n}}
{\sqrt{qac},-q^{\ell}\!\sqrt{qac},q^{-m-2n}}
=\ffnk{ccccc}{q}{q^{\ell}a,\sqrt{qa/c}}{q^{\ell}a/c,\sqrt{qac}}_{\ell}
\\&&\:\:\times\:
\sum_{i=0}^{\ell}\sum_{j=0}^{m}(-1)^jq^{\frac{5i}{2}-nj+\binom{j+1}{2}}\frac{1}{(ac)^{i/2}}
 \frac{1-q^{2\ell-2i}a/c}{1-q^{2\ell}a/c}
\ffnk{ccccc}{q}{q^{-\ell},c,q^{-2\ell}c/a}{q,q^{1-2\ell}/a,q^{1-\ell}c/a}_i
 \\&&\:\:\times\:\ffnk{ccccc}{q}{q^{-m},q^{-m-n},-q^{-m-n},q^{2\ell-i}a,q^ic}
{q,q^{\ell}\!\sqrt{qac},-q^{\ell}\!\sqrt{qac},q^{-m-2n},q^{j-2m-2n-1}}_j
\ffnk{ccccc}{q^2}{q^{1+2\ell-i+j}a,q^{1+i+j}c}{q,q^{1+2\ell+2j}ac}_{m+n-j}.
 \end{eqnarray*}
\end{thm}

\begin{exam}[$\ell=1,m=1$ in Theorem \ref{thm-c}]
\begin{eqnarray*}
&&\xxqdn{_4\phi_3}\ffnk{ccccc}{q;q}{qa,c,q^{-n},-q^{-1-n}}{\sqrt{qac},-\sqrt{q^3ac},q^{-1-2n}}\\
&&\xxqdn\:\:=\:\frac{(1+\sqrt{q^{3+2n}ac})(1-\sqrt{q^{3+2n}a/c})}{(1-\sqrt{qac})(1+\sqrt{qa/c})}\ffnk{ccccc}{q^2}{qa,qc}{q,q^3ac}_{1+n}\\
&&\xxqdn\:\:-\:\,\frac{(1+\sqrt{q^{3+2n}ac})(q^{1+n}-\sqrt{qa/c})}{(1-\sqrt{qac})(1+\sqrt{qa/c})}\ffnk{ccccc}{q^2}{q^2a,c}{q,q^3ac}_{1+n}.
  \end{eqnarray*}
\end{exam}

Letting $\sqrt{a}\to-q^{-\ell}\!\sqrt{a}$ in Theorem \ref{thm-c}, we
obtain the formula
\begin{eqnarray}
&&\xqdn\qqdn{_4\phi_3}\ffnk{ccccc}{q;q}{q^{-\ell}a,c,q^{-n},-q^{-m-n}}
{\sqrt{qac},-q^{-\ell}\!\sqrt{qac},q^{-m-2n}}
=\ffnk{ccccc}{q}{q/a,-\sqrt{qc/a}}{qc/a,-\sqrt{q/ac}}_{\ell}
\nonumber\\\nonumber&&\xqdn\qqdn\:\:\times\:
\sum_{i=0}^{\ell}\sum_{j=0}^{m}(-1)^{i+j}q^{\frac{(1+2\ell)i}{2}-nj+\binom{j+1}{2}}\frac{1}{(ac)^{i/2}}
 \frac{1-q^{2i}c/a}{1-c/a}
\ffnk{ccccc}{q}{q^{-\ell},c,c/a}{q,q/a,q^{1+\ell}c/a}_i
 \nonumber\\\nonumber\\\label{equivalence-f}
  &&\xqdn\qqdn\:\:\times\:\ffnk{ccccc}{q}{q^{-m},q^{-m-n},-q^{-m-n},q^{-i}a,q^ic}
{q,\sqrt{qac},-\sqrt{qac},q^{-m-2n},q^{j-2m-2n-1}}_j
\ffnk{ccccc}{q^2}{q^{1-i+j}a,q^{1+i+j}c}{q,q^{1+2j}ac}_{m+n-j}.
 \end{eqnarray}

It should be pointed out that the corresponding hypergeometric
series identities of Theorem \ref{thm-b} and \eqref{equivalence-e}
are different. This similarly applies to Theorem \ref{thm-c} and
\eqref{equivalence-f}.

\section{Extensions of $q$-Dixon formulas}
\subsection{Extensions of Bailey's $q$-Dixon formula}
\begin{thm}\label{thm-d}
For two complex numbers $\{a,c\}$ and two nonnegative integers
$\{\ell, m\}$, there holds
 \begin{eqnarray*}
&&{_4\phi_3}\ffnk{ccccc}{q;q}{q^{-n},a,c,-q^{1+\ell+m-n}/ac}
{q^{1+\ell-n}/a,q^{1+m-n}/c,-ac}
=\ffnk{ccccc}{q}{q^{1+\ell-n}/a^2,-q/a}{q^{1+\ell}/a^2,-q^{1-n}/a}_{\ell}
\ffnk{ccccc}{q}{-a,q^{-m}c^2}{-ac,q^{-m}c}_{n}
\\&&\:\:\times\:
\sum_{i=0}^{\ell}\sum_{j=0}^{m}(-1)^{i+j}q^{ni+(n-i-j)(j-m)+\binom{j+1}{2}}a^ic^{n-i}
 \frac{1-q^{2i-2\ell-1}a^2}{1-q^{-2\ell-1}a^2}
 \\&&\:\:\times\:
\ffnk{ccccc}{q}{q^{-\ell},q^{-n},q^{-2\ell-1}a^2}{q,q^{n-2\ell}a^2,q^{-\ell}a^2}_i
 \ffnk{ccccc}{q}{q^{-m},q^{i-n},q^{1+2\ell-n-i}/a^2,q^{-m}c,-q^{-m}c}
{q,q^{1+\ell-n}/a,-q^{1+\ell-n}/a,q^{-m}c^2,q^{j-2m-1}c^2}_j
 \\&&\:\:\times\:
\ffnk{ccccc}{q^2}{q,q^{2+2\ell+2m-2n}/a^2c^2}{q^{2+2\ell-2n+2j}/a^2,q^{1-2m+2j}c^2}_{\frac{n-i-j}{2}}\chi(n-i-j).
 \end{eqnarray*}
\end{thm}

\begin{proof}
 Using \eqref{sear} with $a=c,\:b=a,\:c=-q^{1+u+v-n}/ac,\:d=q^{1+v-n}/c,\:e=q^{1+u-n}/a$,
  we get the relation
 \begin{eqnarray}
{_4\phi_3}\ffnk{ccccc}{q;q}{q^{-n},a,c,-q^{1+u+v-n}/ac}
{q^{1+u-n}/a,q^{1+v-n}/c,-ac}&&\xqdn=\:
\ffnk{ccccc}{q}{-a,q^{-v}c^2}{-ac,q^{-v}c}_n
 \nonumber\\\label{dixon-relation-a}&&\xqdn\times\:
 {_4\phi_3}\ffnk{ccccc}{q;q}{q^{-n},q^{1+u-n}/a^2,c,-q^{-v}c}
{q^{1+u-n}/a,-q^{1-n}/a,q^{-v}c^2}.
 \end{eqnarray}
Taking $u=\ell,\,v=m$ in \eqref{dixon-relation-a} and calculating
the $_4\phi_3$-series on the right hand side by
\eqref{equivalence-b}, we establish the theorem.
\end{proof}

\begin{exam}[$\ell=0,m=1$ in Theorem \ref{thm-d}]
\begin{eqnarray*}
&&\xqdn{_4\phi_3}\ffnk{ccccc}{q;q}{q^{-n},a,c,-q^{2-n}/ac}{q^{1-n}/a,q^{2-n}/c,-ac}\\
&&\xqdn\:=\:
\begin{cases}
\ffnk{ccccc}{q}{-a,c^2/q}{-ac,c/q}_{2s}
\ffnk{ccccc}{q^2}{q,q^{2s-2}a^2c^2}{c^2/q,q^{2s}a^2}_{s},&\qqdn
n=2s;
\\[4mm]
\frac{(q-1)c}{q-c^2}\ffnk{ccccc}{q}{-a,c^2/q}{-ac,c/q}_{1+2s}
\ffnk{ccccc}{q^2}{q^3,q^{2s}a^2c^2}{qc^2,q^{2+2s}a^2}_{s}\!,&\qqdn
n=1+2s.
\end{cases}
\end{eqnarray*}
\end{exam}

\begin{exam}[$\ell=1,m=1$ in Theorem \ref{thm-d}]
\begin{eqnarray*}
&&\xqdn{_4\phi_3}\ffnk{ccccc}{q;q}{q^{-n},a,c,-q^{3-n}/ac}{q^{2-n}/a,q^{2-n}/c,-ac}\\
&&\xqdn\:=\:
\begin{cases}
\frac{(q^2+ac)(q^2-q^{2s}a^2)}{(q^2-a^2)(q^2+q^{2s}ac)}
 \ffnk{ccccc}{q}{-a/q,c^2/q}{-ac,c/q}_{2s}
\ffnk{ccccc}{q^2}{q,q^{2s-2}a^2c^2}{c^2/q,q^{2s-2}a^2}_{s},&\qqdn
n=2s;
\\[4mm]
\frac{(a+c)(q-1)(q-q^{2s}ac)}{(q^2-a^2)(1-c^2/q^2)}\ffnk{ccccc}{q}{-a/q,c^2/q^2}{-ac,c/q}_{1+2s}
\ffnk{ccccc}{q^2}{q^3,q^{2s-2}a^2c^2}{c^2/q,q^{2s}a^2}_{s}\!,&\qqdn
n=1+2s.
\end{cases}
\end{eqnarray*}
\end{exam}

Setting $u=\ell,\,v=-m$ in \eqref{dixon-relation-a} and evaluating
the $_4\phi_3$-series on the right hand side by Theorem \ref{thm-a},
we obtain the following theorem.

\begin{thm}\label{thm-e}
For two complex numbers $\{a,c\}$ and two nonnegative integers
$\{\ell, m\}$, there holds
 \begin{eqnarray*}
&&{_4\phi_3}\ffnk{ccccc}{q;q}{q^{-n},a,c,-q^{1+\ell-m-n}/ac}
{q^{1+\ell-n}/a,q^{1-m-n}/c,-ac}
=\ffnk{ccccc}{q}{q^{1+\ell-n}/a^2,-q/a}{q^{1+\ell}/a^2,-q^{1-n}/a}_{\ell}
\ffnk{ccccc}{q}{-a,q^{m}c^2}{-ac,q^{m}c}_{n}
\\&&\:\:\times\:
\sum_{i=0}^{\ell}\sum_{j=0}^{m}(-1)^{i}q^{ni+(m+n-i)j-\binom{j}{2}}a^ic^{n-i}
 \frac{1-q^{2i-2\ell-1}a^2}{1-q^{-2\ell-1}a^2}\\
 &&\:\:\times\:
\ffnk{ccccc}{q}{q^{-\ell},q^{-n},q^{-2\ell-1}a^2}{q,q^{n-2\ell}a^2,q^{-\ell}a^2}_i
 \ffnk{ccccc}{q}{q^{-m},q^{i-n},q^{1+2\ell-n-i}/a^2,c,-c}
{q,q^{1+\ell-n}/a,-q^{1+\ell-n}/a,q^{m}c^2,q^{j-1}c^2}_j
 \\&&\:\:\times\:
\ffnk{ccccc}{q^2}{q,q^{2+2\ell-2n}/a^2c^2}{q^{2+2\ell-2n+2j}/a^2,q^{1+2j}c^2}_{\frac{n-i-j}{2}}\chi(n-i-j).
 \end{eqnarray*}
\end{thm}

\begin{exam}[$\ell=0,m=1$ in Theorem \ref{thm-e}]
\begin{eqnarray*}
&&\xqdn{_4\phi_3}\ffnk{ccccc}{q;q}{q^{-n},a,c,-q^{-n}/ac}{q^{1-n}/a,q^{-n}/c,-ac}\\
&&\xqdn\:=\:
\begin{cases}
\ffnk{ccccc}{q}{-a,qc^2}{-ac,qc}_{2s}
\ffnk{ccccc}{q^2}{q,q^{2s}a^2c^2}{qc^2,q^{2s}a^2}_{s},&\qqdn n=2s;
\\[4mm]
\frac{(1-q)c}{1-qc^2}\ffnk{ccccc}{q}{-a,qc^2}{-ac,qc}_{1+2s}
\ffnk{ccccc}{q^2}{q^3,q^{2+2s}a^2c^2}{q^3c^2,q^{2+2s}a^2}_{s}\!,&\qqdn
n=1+2s.
\end{cases}
\end{eqnarray*}
\end{exam}

\begin{exam}[$\ell=1,m=1$ in Theorem \ref{thm-e}]
\begin{eqnarray*}
&&\xqdn{_4\phi_3}\ffnk{ccccc}{q;q}{q^{-n},a,c,-q^{1-n}/ac}{q^{2-n}/a,q^{-n}/c,-ac}\\
&&\xqdn\:=\:
\begin{cases}
\frac{(q^2-a^2c^2)(1-q^{2s-2}a^2)}{(q^2-a^2)(1-q^{2s-2}a^2c^2)}\ffnk{ccccc}{q}{-a/q,qc^2}{-ac/q,qc}_{2s}
\ffnk{ccccc}{q^2}{q,q^{2s-2}a^2c^2}{qc^2,q^{2s-2}a^2}_{s},&\qqdn
n=2s;
\\[4mm]
\frac{(1-q)(qc-a)}{(q-a)(1-qc)}
\ffnk{ccccc}{q}{-a,qc^2}{-ac,q^2c}_{2s}
\ffnk{ccccc}{q^2}{q^3,q^{2s}a^2c^2}{qc^2,q^{2s}a^2}_{s}\!,&\qqdn
n=1+2s.
\end{cases}
\end{eqnarray*}
\end{exam}

Taking $u=-\ell,\,v=-m$ in \eqref{dixon-relation-a} and calculating
the $_4\phi_3$-series on the right hand side by
\eqref{equivalence-a}, we get the following theorem.

\begin{thm}\label{thm-f}
For two complex numbers $\{a,c\}$ and two nonnegative integers
$\{\ell, m\}$, there holds
\begin{eqnarray*}
&&\xxqdn{_4\phi_3}\ffnk{ccccc}{q;q}{q^{-n},a,c,-q^{1-\ell-m-n}/ac}
{q^{1-\ell-n}/a,q^{1-m-n}/c,-ac}
=\ffnk{ccccc}{q}{a,q^{n}a^2}{a^2,q^{n}a}_{\ell}
\ffnk{ccccc}{q}{-a,q^{m}c^2}{-ac,q^{m}c}_{n}
\\&&\xxqdn\:\:\times\:
\sum_{i=0}^{\ell}\sum_{j=0}^{m}q^{(\ell+n)i+(m+n-i)j-\binom{j}{2}}a^ic^{n-i}
 \frac{1-q^{2i-1}a^2}{1-q^{-1}a^2}
 \\&&\xxqdn\:\:\times\:
\ffnk{ccccc}{q}{q^{-\ell},q^{-n},q^{-1}a^2}{q,q^{n}a^2,q^{\ell}a^2}_i
 \ffnk{ccccc}{q}{q^{-m},q^{i-n},q^{1-n-i}/a^2,c,-c}
{q,q^{1-n}/a,-q^{1-n}/a,q^{m}c^2,q^{j-1}c^2}_j
 \\&&\xxqdn\:\:\times\:
\ffnk{ccccc}{q^2}{q,q^{2-2n}/a^2c^2}{q^{2-2n+2j}/a^2,q^{1+2j}c^2}_{\frac{n-i-j}{2}}\chi(n-i-j).
\end{eqnarray*}
\end{thm}

\begin{exam}[$\ell=1,m=1$ in Theorem \ref{thm-f}]
\begin{eqnarray*}
&&\xxqdn\xqdn\qqdn{_4\phi_3}\ffnk{ccccc}{q;q}{q^{-n},a,c,-q^{-1-n}/ac}{q^{-n}/a,q^{-n}/c,-ac}\\
&&\xxqdn\xqdn\qqdn\:=\:
\begin{cases}
\ffnk{ccccc}{q}{-qa,qc^2}{-qac,qc}_{2s}
\ffnk{ccccc}{q^2}{q,q^{2s+2}a^2c^2}{qc^2,q^{2s+2}a^2}_{s},&\qqdn
n=2s;
\\[4mm]
\frac{a+c}{1+ac} \ffnk{ccccc}{q}{-qa,qc^2}{-qac,qc}_{1+2s}
\ffnk{ccccc}{q^2}{q,q^{2s+2}a^2c^2}{qc^2,q^{2s+2}a^2}_{1+s}\!,&\qqdn
n=1+2s.
\end{cases}
\end{eqnarray*}
\end{exam}

Letting $u=-\ell,\,v=m$ in \eqref{dixon-relation-a} and evaluating
the $_4\phi_3$-series on the right hand side by
\eqref{equivalence-c}, we can derive summation formula for the
following series:
 \begin{eqnarray*}
{_4\phi_3}\ffnk{ccccc}{q;q}{q^{-n},a,c,-q^{1-\ell+m-n}/ac}
{q^{1-\ell-n}/a,q^{1+m-n}/c,-ac},
 \end{eqnarray*}
which is equivalent to Theorem \ref{thm-e}. The corresponding
concrete result has been omitted.

\subsection{Extensions of another $q$-Dixon formula}
\begin{thm}\label{thm-g}
For two complex numbers $\{a,c\}$ and two nonnegative integers
$\{\ell, m\}$, there holds
 \begin{eqnarray*}
&&{_4\phi_3}\ffnk{ccccc}{q;q}{a,c,q^{-n},-q^{1+\ell+m+n}a/c}
{q^{1+\ell}a/c,q^{1+m+n}a,-q^{-n}c}
=\ffnk{ccccc}{q}{a,-q^{-\ell}c}{q^{-\ell-1}c^2,-qa/c}_{\ell}
\ffnk{ccccc}{q}{q^{1+m+n},-qa/c}{q^{1+m+n}a,-q/c}_{n}
\\&&\:\:\times\:
\sum_{i=0}^{\ell}\sum_{j=0}^{m}(-1)^{i+j}q^{\binom{j+1}{2}-nj}\Big(\frac{c}{a}\Big)^i
 \frac{1-q^{1+2i}/c^2}{1-q/c^2}\ffnk{ccccc}{q}{q^{-\ell},q^{1+\ell}a/c^2,q/c^2}{q,q^{1-\ell}/a,q^{2+\ell}/c^2}_i
 \\&&\:\:\times\:
 \ffnk{ccccc}{q}{q^{-m},q^{-m-n},-q^{-m-n},q^{\ell-i}a,q^{1+\ell+i}a/c^2}
{q,q^{1+\ell}a/c,-q^{1+\ell}a/c,q^{-m-2n},q^{j-2m-2n-1}}_j
\\&&\:\:\times\:
\ffnk{ccccc}{q^2}{q^{1+\ell-i+j}a,q^{2+\ell+i+j}a/c^2}{q,q^{2+2\ell+2j}a^2/c^2}_{m+n-j}.
 \end{eqnarray*}
\end{thm}

\begin{proof}
Utilizing \eqref{sear} with
$a=a,\:b=c,\:c=-q^{1+u+v+n}a/c,\:d=q^{1+v+n}a,\:e=q^{1+u}a/c$, we
obtain the relation
 \begin{eqnarray}
{_4\phi_3}\ffnk{ccccc}{q;q}{a,c,q^{-n},-q^{1+u+v+n}a/c}
{q^{1+u}a/c,q^{1+v+n}a,-q^{-n}c}&&\xqdn=\:
\ffnk{ccccc}{q}{q^{1+v+n},-qa/c}{q^{1+v+n}a,-q/c}_n
 \nonumber\\\label{dixon-relation-b}&&\xqdn\times\:
 {_4\phi_3}\ffnk{ccccc}{q;q}{a,q^{1+u}a/c^2,q^{-n},-q^{-v-n}}
{q^{1+u}a/c,-qa/c,q^{-v-2n}}.
\end{eqnarray}
Setting $u=\ell,\,v=m$ in \eqref{dixon-relation-b} and evaluating
the $_4\phi_3$-series on the right hand side by Theorem \ref{thm-c},
we establish the theorem.
\end{proof}

\begin{exam}[$\ell=0,m=1$ in Theorem \ref{thm-g}]
\begin{eqnarray*}
&&\xxqdn{_4\phi_3}\ffnk{ccccc}{q;q}{a,c,q^{-n},-q^{2+n}a/c}{qa/c,q^{2+n}a,-q^{-n}c}\\
&&\xxqdn\:\:=\:\ffnk{ccccc}{q}{q^{2+n},-qa/c}{q^{2+n}a,-q/c}_{n}\ffnk{ccccc}{q^2}{qa,q^2a/c^2}{q,q^2a^2/c^2}_{1+n}\\
&&\xxqdn\:\:-\:\:q^{1+n}\ffnk{ccccc}{q}{q^{2+n},-qa/c}{q^{2+n}a,-q/c}_{n}\ffnk{ccccc}{q^2}{a,qa/c^2}{q,q^2a^2/c^2}_{1+n}.
\end{eqnarray*}
\end{exam}

\begin{exam}[$\ell=1,m=0$ in Theorem \ref{thm-g}]
 \begin{eqnarray*}
&&\xxqdn\xqdn\qqdn{_4\phi_3}\ffnk{ccccc}{q;q}{a,c,q^{-n},-q^{2+n}a/c}{q^2a/c,q^{1+n}a,-q^{-n}c}\\
&&\xxqdn\xqdn\qqdn\:\:=\:\frac{qc(1-a)}{(qa+c)(q-c)}\ffnk{ccccc}{q}{q^{1+n},-qa/c}{q^{1+n}a,-q/c}_{n}
\ffnk{ccccc}{q^2}{q^2a,q^3a/c^2}{q,q^4a^2/c^2}_{n}\\
&&\xxqdn\xqdn\qqdn\:\:+\:\:\frac{q^2a-c^2}{(qa+c)(q-c)}\ffnk{ccccc}{q}{q^{1+n},-qa/c}{q^{1+n}a,-q/c}_{n}
\ffnk{ccccc}{q^2}{qa,q^4a/c^2}{q,q^4a^2/c^2}_{n}.
 \end{eqnarray*}
\end{exam}

\begin{exam}[$\ell=1,m=1$ in Theorem \ref{thm-g}]
\begin{eqnarray*}
&&\xqdn{_4\phi_3}\ffnk{ccccc}{q;q}{a,c,q^{-n},-q^{3+n}a/c}{q^2a/c,q^{2+n}a,-q^{-n}c}\\
&&\xqdn\:\:=\:\frac{q(1+q^{n}c)(c-q^{2+n}a)}{(qa+c)(q-c)}\ffnk{ccccc}{q}{q^{2+n},-qa/c}{q^{2+n}a,-q/c}_{n}
\ffnk{ccccc}{q^2}{a,q^3a/c^2}{q,q^4a^2/c^2}_{1+n}\\
&&\xqdn\:\:-\:\:\frac{(c+q^{2+n})(c-q^{2+n}a)}{(qa+c)(q-c)}\ffnk{ccccc}{q}{q^{2+n},-qa/c}{q^{2+n}a,-q/c}_{n}
\ffnk{ccccc}{q^2}{qa,q^2a/c^2}{q,q^4a^2/c^2}_{1+n}.
  \end{eqnarray*}
\end{exam}

Taking $u=\ell,\,v=-m$ in \eqref{dixon-relation-b} and calculating
the $_4\phi_3$-series on the right hand side by Theorem \ref{thm-b},
we get the following theorem.

\begin{thm}\label{thm-h}
For two complex numbers $\{a,c\}$ and two nonnegative integers
$\{\ell, m\}$ with $m\leq n$, there holds
 \begin{eqnarray*}
&&\xqdn{_4\phi_3}\ffnk{ccccc}{q;q}{a,c,q^{-n},-q^{1+\ell-m+n}a/c}
{q^{1+\ell}a/c,q^{1-m+n}a,-q^{-n}c}
=\ffnk{ccccc}{q}{a,-q^{-\ell}c}{q^{-\ell-1}c^2,-qa/c}_{\ell}
\ffnk{ccccc}{q}{q^{1-m+n},-qa/c}{q^{1-m+n}a,-q/c}_{n}
\\&&\xqdn\:\:\times\:
\sum_{i=0}^{\ell}\sum_{j=0}^{m}(-1)^{i}q^{(m-n)j+\binom{j+1}{2}}\Big(\frac{c}{a}\Big)^i
 \frac{1-q^{1+2i}/c^2}{1-q/c^2}\ffnk{ccccc}{q}{q^{-\ell},q^{1+\ell}a/c^2,q/c^2}{q,q^{1-\ell}/a,q^{2+\ell}/c^2}_i
 \\&&\xqdn\:\:\times\:
 \ffnk{ccccc}{q}{q^{-m},q^{-n},-q^{-n},q^{\ell-i}a,q^{1+\ell+i}a/c^2}
{q,q^{1+\ell}a/c,-q^{1+\ell}a/c,q^{m-2n},q^{j-2n-1}}_j
\ffnk{ccccc}{q^2}{q^{1+\ell-i+j}a,q^{2+\ell+i+j}a/c^2}{q,q^{2+2\ell+2j}a^2/c^2}_{n-j}.
  \end{eqnarray*}
\end{thm}

\begin{exam}[$\ell=0,m=1$ in Theorem \ref{thm-h}: $n\geq1$]
 \begin{eqnarray*}
&&\xxqdn\xxqdn\qqdn{_4\phi_3}\ffnk{ccccc}{q;q}{a,c,q^{-n},-q^{n}a/c}{qa/c,q^{n}a,-q^{-n}c}\\
&&\xxqdn\xxqdn\qqdn\:\:=\:\ffnk{ccccc}{q}{q^{n},-qa/c}{q^{n}a,-q/c}_{n}\ffnk{ccccc}{q^2}{qa,q^2a/c^2}{q,q^2a^2/c^2}_{n}\\
&&\xxqdn\xxqdn\qqdn\:\:+\:\:q^{n}\ffnk{ccccc}{q}{q^{n},-qa/c}{q^{n}a,-q/c}_{n}\ffnk{ccccc}{q^2}{a,qa/c^2}{q,q^2a^2/c^2}_{n}.
  \end{eqnarray*}
\end{exam}

\begin{exam}[$\ell=1,m=1$ in Theorem \ref{thm-h}: $n\geq1$]
 \begin{eqnarray*}
&&\xxqdn{_4\phi_3}\ffnk{ccccc}{q;q}{a,c,q^{-n},-q^{1+n}a/c}{q^2a/c,q^{n}a,-q^{-n}c}\\
&&\xxqdn\:\:=\:\frac{q-q^nc}{q-c}\ffnk{ccccc}{q}{q^{n},-q^2a/c}{q^{n}a,-q/c}_{n}\ffnk{ccccc}{q^2}{a,q^3a/c^2}{q,q^4a^2/c^2}_{n}\\
&&\xxqdn\:\:+\:\:\frac{q^{1+n}-c}{q-c}\ffnk{ccccc}{q}{q^{n},-q^2a/c}{q^{n}a,-q/c}_{n}\ffnk{ccccc}{q^2}{qa,q^2a/c^2}{q,q^4a^2/c^2}_{n}.
  \end{eqnarray*}
\end{exam}

Setting $u=-\ell,\,v=m$ in \eqref{dixon-relation-b} and evaluating
the $_4\phi_3$-series on the right hand side by
\eqref{equivalence-f}, we obtain the following theorem.

\begin{thm}\label{thm-i}
For two complex numbers $\{a,c\}$ and two nonnegative integers
$\{\ell, m\}$, there holds
 \begin{eqnarray*}
&&\xqdn\qdn{_4\phi_3}\ffnk{ccccc}{q;q}{a,c,q^{-n},-q^{1-\ell+m+n}a/c}
{q^{1-\ell}a/c,q^{1+m+n}a,-q^{-n}c}
=\ffnk{ccccc}{q}{a,c}{q^{\ell-1}c^2,q^{1-\ell}a/c}_{\ell}
\ffnk{ccccc}{q}{q^{1+m+n},-qa/c}{q^{1+m+n}a,-q/c}_{n}
\\&&\xqdn\qdn\:\:\times\:
\sum_{i=0}^{\ell}\sum_{j=0}^{m}(-1)^{j}q^{\ell i-nj+\binom{j+1}{2}}
\Big(\frac{c}{a}\Big)^i
 \frac{1-q^{1-2\ell+2i}/c^2}{1-q^{1-2\ell}/c^2}\ffnk{ccccc}{q}{q^{-\ell},q^{1-\ell}a/c^2,q^{1-2\ell}/c^2}{q,q^{1-\ell}/a,q^{2-\ell}/c^2}_i
 \\&&\xqdn\qdn\:\:\times\:
 \ffnk{ccccc}{q}{q^{-m},q^{-m-n},-q^{-m-n},q^{\ell-i}a,q^{1-\ell+i}a/c^2}
{q,qa/c,-qa/c,q^{-m-2n},q^{j-2m-2n-1}}_j
\ffnk{ccccc}{q^2}{q^{1+\ell-i+j}a,q^{2-\ell+i+j}a/c^2}{q,q^{2+2j}a^2/c^2}_{m+n-j}.
\end{eqnarray*}
\end{thm}

\begin{exam}[$\ell=1,m=0$ in Theorem \ref{thm-i}]
 \begin{eqnarray*}
&&\xxqdn\xqdn\qqdn{_4\phi_3}\ffnk{ccccc}{q;q}{a,c,q^{-n},-q^{n}a/c}{a/c,q^{1+n}a,-q^{-n}c}\\
&&\xxqdn\xqdn\qqdn\:\:=\:\frac{(1-a)c}{(1+c)(c-a)}\ffnk{ccccc}{q}{q^{1+n},-qa/c}{q^{1+n}a,-q/c}_{n}
\ffnk{ccccc}{q^2}{q^2a,qa/c^2}{q,q^2a^2/c^2}_{n}\\
&&\xxqdn\xqdn\qqdn\:\:+\:\:\frac{c^2-a}{(1+c)(c-a)}\ffnk{ccccc}{q}{q^{1+n},-qa/c}{q^{1+n}a,-q/c}_{n}
\ffnk{ccccc}{q^2}{qa,q^2a/c^2}{q,q^2a^2/c^2}_{n}.
\end{eqnarray*}
\end{exam}

\begin{exam}[$\ell=1,m=1$ in Theorem \ref{thm-i}]
 \begin{eqnarray*}
&&\xqdn{_4\phi_3}\ffnk{ccccc}{q;q}{a,c,q^{-n},-q^{1+n}a/c}{a/c,q^{2+n}a,-q^{-n}c}\\
&&\xqdn\:\:=\:\frac{(qa+c)(1-q^{1+n}c)}{(1+c)(c-a)}\ffnk{ccccc}{q}{q^{2+n},-q^2a/c}{q^{2+n}a,-q/c}_{n}
\ffnk{ccccc}{q^2}{a,qa/c^2}{q,q^2a^2/c^2}_{1+n}\\
&&\xqdn\:\:+\:\:\frac{(qa+c)(c-q^{1+n})}{(1+c)(c-a)}\ffnk{ccccc}{q}{q^{2+n},-q^2a/c}{q^{2+n}a,-q/c}_{n}
\ffnk{ccccc}{q^2}{qa,a/c^2}{q,q^2a^2/c^2}_{1+n}.
\end{eqnarray*}
\end{exam}

Taking $u=-\ell,\,v=-m$ in \eqref{dixon-relation-b} and calculating
the $_4\phi_3$-series on the right hand side by
\eqref{equivalence-e}, we get the following theorem.

\begin{thm}\label{thm-j}
For two complex numbers $\{a,c\}$ and two nonnegative integers
$\{\ell, m\}$ with $m\leq n$, there holds
\begin{eqnarray*}
&&\xqdn\qdn{_4\phi_3}\ffnk{ccccc}{q;q}{a,c,q^{-n},-q^{1-\ell-m+n}a/c}
{q^{1-\ell}a/c,q^{1-m+n}a,-q^{-n}c}
=\ffnk{ccccc}{q}{a,c}{q^{\ell-1}c^2,q^{1-\ell}a/c}_{\ell}
\ffnk{ccccc}{q}{q^{1-m+n},-qa/c}{q^{1-m+n}a,-q/c}_{n}
\\&&\xqdn\qdn\:\:\times\:
\sum_{i=0}^{\ell}\sum_{j=0}^{m}q^{\ell i+(m-n)j+\binom{j+1}{2}}
\Big(\frac{c}{a}\Big)^i
 \frac{1-q^{1-2\ell+2i}/c^2}{1-q^{1-2\ell}/c^2}\ffnk{ccccc}{q}{q^{-\ell},q^{1-\ell}a/c^2,q^{1-2\ell}/c^2}{q,q^{1-\ell}/a,q^{2-\ell}/c^2}_i
 \\&&\xqdn\qdn\:\:\times\:
 \ffnk{ccccc}{q}{q^{-m},q^{-n},-q^{-n},q^{\ell-i}a,q^{1-\ell+i}a/c^2}
{q,qa/c,-qa/c,q^{m-2n},q^{j-2n-1}}_j\ffnk{ccccc}{q^2}{q^{1+\ell-i+j}a,q^{2-\ell+i+j}a/c^2}{q,q^{2+2j}a^2/c^2}_{n-j}.
 \end{eqnarray*}
\end{thm}

\begin{exam}[$\ell=1,m=1$ in Theorem \ref{thm-j}: $n\geq1$]
 \begin{eqnarray*}
&&\xqdn{_4\phi_3}\ffnk{ccccc}{q;q}{a,c,q^{-n},-q^{n-1}a/c}{a/c,q^{n}a,-q^{-n}c}\\
&&\xqdn\:\:=\:\ffnk{ccccc}{q}{q^{n},-a/c}{q^{n}a,-1/c}_{n}
\ffnk{ccccc}{q^2}{qa,a/c^2}{q,a^2/c^2}_{n}\\
&&\xqdn\:\:+\:\:\frac{1+q^nc}{1+c}\ffnk{ccccc}{q}{q^{n},-a/c}{q^{n}a,-q/c}_{n}
\ffnk{ccccc}{q^2}{a,qa/c^2}{q,a^2/c^2}_{n}.
 \end{eqnarray*}
\end{exam}

\section{Extensions of $q$-Whipple formulas}
\subsection{Extensions of Andrews' $q$-Whipple formula}
\begin{thm}\label{thm-k}
For two complex numbers $\{a,c\}$ and two nonnegative integers
$\{\ell, m\}$, there holds
\begin{eqnarray*}
&&\qdn\xqdn{_4\phi_3}\ffnk{ccccc}{q;q}{q^{-n},q^{1+\ell+n},\sqrt{qac},-q^{m}\!\sqrt{qac}}
{-q,q^{1+\ell+m}a,qc}
=\ffnk{ccccc}{q}{q^{-n}\!\sqrt{qa/c},-q^{m-n}\!\sqrt{qa/c},q^{1+\ell+n}}{q^{1+\ell+m}a,q^{-n}/c,-q}_{n}
\\&&\qdn\xqdn\:\:\times\:
\ffnk{ccccc}{q}{q^{m-n}a,\sqrt{qac}}{q^{m}ac,q^{-n}\!\sqrt{qa/c}}_{m}
\sum_{i=0}^{\ell}\sum_{j=0}^{m}(-1)^iq^{(n+\frac{5}{2})j-ni+\binom{i+1}{2}}\Big(\frac{c}{a}\Big)^{\frac{j}{2}}
\frac{1-q^{2m-2j}ac}{1-q^{2m}ac}
 \\&&\qdn\xqdn\:\:\times\:
\ffnk{ccccc}{q}{q^{-\ell},q^{-\ell-n},-q^{-\ell-n},q^{2m-n-j}a,q^{j-n}/c}
{q,q^{m-n}\!\!\sqrt{qa/c},-q^{m-n}\!\!\sqrt{qa/c},q^{-\ell-2n},q^{i-2\ell-2n-1}}_i
\ffnk{ccccc}{q}{q^{-m},q^{-n}/c,q^{-2m}/ac}{q,q^{1-2m+n}/a,q^{1-m}/ac}_j
\\&&\qdn\xqdn\:\:\times\:
\ffnk{ccccc}{q^2}{q^{1+2m-n+i-j}a,q^{1-n+i+j}/c}{q,q^{1+2m-2n+2i}a/c}_{\ell+n-i}.
\end{eqnarray*}
\end{thm}

\begin{proof}
The iteration of \eqref{sear} produces the transformation formula:
 \begin{eqnarray*}
\quad{_4\phi_3}\ffnk{ccccc}{q;q}{q^{-n},a,b,c}{d,e,q^{1-n}abc/de}
=\ffnk{ccccc}{q}{c,de/ac,de/bc}{d,e,de/abc}_n
{_4\phi_3}\ffnk{ccccc}{q;q}{q^{-n},d/c,e/c,de/abc}{de/ac,de/bc,q^{1-n}/c}.
 \end{eqnarray*}
 Using the last equation with $a=\sqrt{qac},\:b=-q^{v}\!\sqrt{qac},\:c=q^{1+n+u},\:d=q^{1+u+v}a,\:e=-q$, we obtain the relation
\begin{eqnarray}
&&{_4\phi_3}\ffnk{ccccc}{q;q}{q^{-n},q^{1+n+u},\sqrt{qac},-q^{v}\!\sqrt{qac}}
{-q,q^{1+u+v}a,qc}
 \nonumber\\\nonumber&&\:\,=\:
\ffnk{ccccc}{q}{q^{-n}\!\sqrt{qa/c},-q^{v-n}\!\sqrt{qa/c},q^{1+n+u}}{q^{1+u+v}a,q^{-n}/c,-q}_{n}
 \\\label{whipple-relation-a}&&\:\:\times\:
 {_4\phi_3}\ffnk{ccccc}{q;q}{q^{v-n}a,q^{-n}/c,q^{-n},-q^{-u-n}}
{q^{-n}\!\sqrt{qa/c},-q^{v-n}\!\sqrt{qa/c},q^{-u-2n}}.
 \end{eqnarray}
Setting $u=\ell,\,v=m$ in \eqref{whipple-relation-a} and evaluating
the $_4\phi_3$-series on the right hand side by Theorem \ref{thm-c},
we establish the theorem.
\end{proof}

\begin{exam}[$\ell=0,m=1$ in Theorem \ref{thm-k}]
 \begin{eqnarray*}
&&\xqdn{_4\phi_3}\ffnk{ccccc}{q;q}{q^{-n},q^{1+n},\sqrt{qac},-\sqrt{q^3ac}}{-q,q^{2}a,qc}\\
&&\xqdn\:\:=\:\frac{q^{\binom{1+n}{2}}}{(1+\sqrt{qac})(1-\sqrt{qa/c})}
\frac{(q^{1-n}a;q^2)_{1+n}(q^{1-n}c;q^2)_{n}}{(q^2a;q)_n(qc;q)_n}\\
&&\xqdn\:\:+\:\:\frac{q^{\binom{1+n}{2}}}{(1+\sqrt{qac})(1-\sqrt{c/qa})}
\frac{(q^{2-n}a;q^2)_{n}(q^{-n}c;q^2)_{1+n}}{(q^2a;q)_n(qc;q)_n}.
\end{eqnarray*}
\end{exam}

\begin{exam}[$\ell=1,m=0$ in Theorem \ref{thm-k}]
 \begin{eqnarray*}
&&\xqdn{_4\phi_3}\ffnk{ccccc}{q;q}{q^{-n},q^{2+n},\sqrt{qac},-\sqrt{qac}}{-q,q^{2}a,qc}\\
&&\xqdn\:\:=\:\frac{q^{\binom{2+n}{2}}}{(qa-c)(1-q^{1+n})}
\frac{(q^{1-n}a;q^2)_{1+n}(q^{-1-n}c;q^2)_{1+n}}{(q^2a;q)_n(qc;q)_n}\\
&&\xqdn\:\:-\:\:\frac{q^{\binom{2+n}{2}}}{(qa-c)(1-q^{1+n})}
\frac{(q^{-n}a;q^2)_{1+n}(q^{-n}c;q^2)_{1+n}}{(q^2a;q)_n(qc;q)_n}.
 \end{eqnarray*}
\end{exam}

\begin{exam}[$\ell=1,m=1$ in Theorem \ref{thm-k}]
  \begin{eqnarray*}
&&\xqdn{_4\phi_3}\ffnk{ccccc}{q;q}{q^{-n},q^{2+n},\sqrt{qac},-\sqrt{q^3ac}}{-q,q^{3}a,qc}\\
&&\xqdn\:\:=\:\tfrac{q^{\binom{2+n}{2}}(1-\sqrt{q^{-1-2n}ac})}{(\sqrt{qa}-\sqrt{c})(\sqrt{q^3a}-\sqrt{c})(1+\sqrt{qac})(1-q^{1+n})}
\frac{(q^{2-n}a;q^2)_{1+n}(q^{-n}c;q^2)_{1+n}}{(q^3a;q)_n(qc;q)_n}\\
&&\xqdn\:\:-\:\:\tfrac{q^{\binom{2+n}{2}}(1-\sqrt{q^{3+2n}ac})}{(\sqrt{qa}-\sqrt{c})(\sqrt{q^3a}-\sqrt{c})(1+\sqrt{qac})(1-q^{1+n})}
\frac{(q^{1-n}a;q^2)_{1+n}(q^{-1-n}c;q^2)_{1+n}}{(q^3a;q)_n(qc;q)_n}.
 \end{eqnarray*}
\end{exam}

Taking $u=-\ell,\,v=m$ in \eqref{whipple-relation-a} and calculating
the $_4\phi_3$-series on the right hand side by Theorem \ref{thm-b},
we get the following theorem.

\begin{thm}\label{thm-l}
For two complex numbers $\{a,c\}$ and two nonnegative integers
$\{\ell, m\}$ with $\ell\leq n$, there holds
  \begin{eqnarray*}
&&\qdn\xqdn{_4\phi_3}\ffnk{ccccc}{q;q}{q^{-n},q^{1-\ell+n},\sqrt{qac},-q^{m}\!\sqrt{qac}}
{-q,q^{1-\ell+m}a,qc}
=\ffnk{ccccc}{q}{q^{-n}\!\sqrt{qa/c},-q^{m-n}\!\sqrt{qa/c},q^{1-\ell+n}}{q^{1-\ell+m}a,q^{-n}/c,-q}_{n}
\\&&\qdn\xqdn\:\:\times\:
\ffnk{ccccc}{q}{q^{m-n}a,\sqrt{qac}}{q^{m}ac,q^{-n}\!\sqrt{qa/c}}_{m}
\sum_{i=0}^{\ell}\sum_{j=0}^{m}q^{(\ell-n)i+\binom{i+1}{2}+(n+\frac{5}{2})j}\Big(\frac{c}{a}\Big)^{\frac{j}{2}}
\frac{1-q^{2m-2j}ac}{1-q^{2m}ac}
 \\&&\qdn\xqdn\:\:\times\:
\ffnk{ccccc}{q}{q^{-\ell},q^{-n},-q^{-n},q^{2m-n-j}a,q^{j-n}/c}
{q,q^{m-n}\!\!\sqrt{qa/c},-q^{m-n}\!\!\sqrt{qa/c},q^{\ell-2n},q^{i-2n-1}}_i
\ffnk{ccccc}{q}{q^{-m},q^{-n}/c,q^{-2m}/ac}{q,q^{1-2m+n}/a,q^{1-m}/ac}_j
 \\&&\qdn\xqdn\:\:\times\:
 \ffnk{ccccc}{q^2}{q^{1+2m-n+i-j}a,q^{1-n+i+j}/c}{q,q^{1+2m-2n+2i}a/c}_{n-i}.
 \end{eqnarray*}
\end{thm}

\begin{exam}[$\ell=1,m=0$ in Theorem \ref{thm-l}]
\begin{eqnarray*}
&&\xxqdn\xxqdn\xqdn\qqdn{_4\phi_3}\ffnk{ccccc}{q;q}{q^{-n},q^{n},\sqrt{qac},-\sqrt{qac}}{-q,a,qc}\\
&&\xxqdn\xxqdn\xqdn\qqdn\:\:=\:\frac{q^{\binom{1+n}{2}}}{1+q^{n}}
\frac{(q^{1-n}a;q^2)_{n}(q^{1-n}c;q^2)_{n}}{(a;q)_n(qc;q)_n}\\
&&\xxqdn\xxqdn\xqdn\qqdn\:\:+\:\:\frac{q^{\binom{1+n}{2}}}{1+q^{n}}
\frac{(q^{-n}a;q^2)_{n}(q^{2-n}c;q^2)_{n}}{(a;q)_n(qc;q)_n}.
 \end{eqnarray*}
\end{exam}

\begin{exam}[$\ell=1,m=1$ in Theorem \ref{thm-l}]
\begin{eqnarray*}
&&\xqdn{_4\phi_3}\ffnk{ccccc}{q;q}{q^{-n},q^{n},\sqrt{qac},-\sqrt{q^3ac}}{-q,qa,qc}\\
&&\xqdn\:\:=\:\frac{q^{\binom{1+n}{2}}(1+\sqrt{q^{1+2n}ac})}{(1+q^{n})(1+\sqrt{qac})}
\frac{(q^{1-n}a;q^2)_{n}(q^{1-n}c;q^2)_{n}}{(qa;q)_n(qc;q)_n}\\
&&\xqdn\:\:+\:\:\frac{q^{\binom{1+n}{2}}(1+\sqrt{q^{1-2n}ac})}{(1+q^{n})(1+\sqrt{qac})}
\frac{(q^{2-n}a;q^2)_{n}(q^{2-n}c;q^2)_{n}}{(qa;q)_n(qc;q)_n}.
 \end{eqnarray*}
\end{exam}

Performing the replacement $\sqrt{a}\to-q^{-m}\!\sqrt{a}$ in
Theorems \ref{thm-k} and \ref{thm-l}, we can derive summation
formulas for the following two series:
 \begin{eqnarray*}
&&\qdn\xqdn{_4\phi_3}\ffnk{ccccc}{q;q}{q^{-n},q^{1+\ell+n},\sqrt{qac},-q^{-m}\!\sqrt{qac}}
{-q,q^{1+\ell-m}a,qc},\\
&&\qdn\xqdn{_4\phi_3}\ffnk{ccccc}{q;q}{q^{-n},q^{1-\ell+n},\sqrt{qac},-q^{-m}\!\sqrt{qac}}
{-q,q^{1-\ell-m}a,qc}.
\end{eqnarray*}
The corresponding concrete results will not be displayed here.

\subsection{Extensions of Jain's $q$-Whipple formula}

\begin{thm}\label{thm-m}
For two complex numbers $\{a,c\}$ and two nonnegative integers
$\{\ell, m\}$ with $m\leq n$, there holds
  \begin{eqnarray*}
&&\xqdn{_4\phi_3}\ffnk{ccccc}{q;q}{a,q^{1+\ell}/a,q^{-n},-q^{m-n}}{c,q^{1+\ell+m-2n}/c,-q}
=\ffnk{ccccc}{q}{q^{2}/ac,-q/a}{q^{2}/a^2,-q/c}_{\ell}
\ffnk{ccccc}{q}{q^{1-m+n},-q^{-\ell}c}{q^{n-m-\ell}c,-q}_{n}
\\&&\xqdn\:\:\times\:
\sum_{i=0}^{\ell}\sum_{j=0}^{m}q^{(\ell+1)i+(m-n)j+\binom{j+1}{2}}\frac{(-1)^i}{c^i}
 \frac{1-q^{1+2i}/a^2}{1-q/a^2}\ffnk{ccccc}{q}{q^{-\ell},c/a,q/a^2}{q,q^2/ac,q^{2+\ell}/a^2}_i
 \\&&\xqdn\:\:\times\:
 \ffnk{ccccc}{q}{q^{-m},q^{-n},-q^{-n},q^{-i-1}ac,q^{i}c/a}{q,c,-c,q^{m-2n},q^{j-2n-1}}_j
\ffnk{ccccc}{q^2}{q^{j-i}ac,q^{1+i+j}c/a}{q,q^{2j}c^2}_{n-j}.
 \end{eqnarray*}
\end{thm}

\begin{proof}
 Utilizing \eqref{sear} with $a=-q^{v-n},\:b=a,\:c=q^{1+u}/a,\:d=-q,\:e=c$, we obtain the relation
 \begin{eqnarray}
{_4\phi_3}\ffnk{ccccc}{q;q}{a,q^{1+u}/a,q^{-n},-q^{v-n}}
{c,q^{1+u+v-2n}/c,-q}&&\xqdn=\:
\ffnk{ccccc}{q}{q^{1-v+n},-q^{-u}c}{q^{n-u-v}c,-q}_{n}
 \nonumber\\\label{whipple-relation-b}&&\xqdn\times\:
 {_4\phi_3}\ffnk{ccccc}{q;q}{q^{-u-1}ac,c/a,q^{-n},-q^{v-n}}
{c,-q^{-u}c,q^{v-2n}}.
 \end{eqnarray}
Setting $u=\ell,\,v=m$ in \eqref{whipple-relation-b} and evaluating
the $_4\phi_3$-series on the right hand side by
\eqref{equivalence-e}, we establish the theorem.
\end{proof}

\begin{exam}[$\ell=0,m=1$ in Theorem \ref{thm-m}]
 \begin{eqnarray*}
&&\xxqdn\xxqdn\xqdn\qdn{_4\phi_3}\ffnk{ccccc}{q;q}{a,q/a,q^{-n},-q^{1-n}}{c,q^{2-2n}/c,-q}\\
&&\xxqdn\xxqdn\xqdn\qdn\:\:=\:
\frac{1}{1+q^n}\frac{(ac;q^2)_{n}(qc/a;q^2)_{n}}{(c;q)_n(q^{n-1}c;q)_n}
\\&&\xxqdn\xxqdn\xqdn\qdn\:\:+\:\:\frac{q^n}{1+q^n}\frac{(ac/q;q^2)_{n}(c/a;q^2)_{n}}{(c;q)_n(q^{n-1}c;q)_n}.
 \end{eqnarray*}
\end{exam}

\begin{exam}[$\ell=1,m=0$ in Theorem \ref{thm-m}]
 \begin{eqnarray*}
&&\xqdn{_4\phi_3}\ffnk{ccccc}{q;q}{a,q^2/a,q^{-n},-q^{-n}}{c,q^{2-2n}/c,-q}\\
&&\xqdn\:\:=\:
\frac{(q-c)(q^2-ac)}{(q-a)(q^2-q^{2n}c^2)}\frac{(ac;q^2)_{n}(qc/a;q^2)_{n}}{(c/q;q)_{2n}}
\\&&\xqdn\:\:+\:\:\frac{q(q-c)(c-a)}{(q-a)(q^2-q^{2n}c^2)}\frac{(ac/q;q^2)_{n}(q^2c/a;q^2)_{n}}{(c/q;q)_{2n}}.
 \end{eqnarray*}
\end{exam}

\begin{exam}[$\ell=1,m=1$ in Theorem \ref{thm-m}]
 \begin{eqnarray*}
&&\xxqdn\!\!{_4\phi_3}\ffnk{ccccc}{q;q}{a,q^2/a,q^{-n},-q^{1-n}}{c,q^{3-2n}/c,-q}\\
&&\xxqdn\:=\:
\frac{q-q^na}{(q-a)(1+q^n)}\frac{(ac/q^2;q^2)_{n}(qc/a;q^2)_{n}}{(c;q)_{n}(q^{n-2}c;q)_{n}}
\\&&\xxqdn\:+\:\:\frac{q^{1+n}-a}{(q-a)(1+q^n)}\frac{(ac/q;q^2)_{n}(c/a;q^2)_{n}}{(c;q)_{n}(q^{n-2}c;q)_{n}}.
 \end{eqnarray*}
\end{exam}

Taking $u=\ell,\,v=-m$ in \eqref{whipple-relation-b} and calculating
the $_4\phi_3$-series on the right hand side by
\eqref{equivalence-f}, we get the following theorem.

\begin{thm}\label{thm-n}
For two complex numbers $\{a,c\}$ and two nonnegative integers
$\{\ell, m\}$, there holds
 \begin{eqnarray*}
&&\xqdn{_4\phi_3}\ffnk{ccccc}{q;q}{a,q^{1+\ell}/a,q^{-n},-q^{-m-n}}
{c,q^{1+\ell-m-2n}/c,-q}
=\ffnk{ccccc}{q}{q^{2}/ac,-q/a}{q^{2}/a^2,-q/c}_{\ell}
\ffnk{ccccc}{q}{q^{1+m+n},-q^{-\ell}c}{q^{n+m-\ell}c,-q}_{n}
\\&&\xqdn\:\:\times\:
\sum_{i=0}^{\ell}\sum_{j=0}^{m}q^{(\ell+1)i-nj+\binom{j+1}{2}}\frac{(-1)^{i+j}}{c^i}
 \frac{1-q^{1+2i}/a^2}{1-q/a^2}\ffnk{ccccc}{q}{q^{-\ell},c/a,q/a^2}{q,q^2/ac,q^{2+\ell}/a^2}_i
 \\&&\xqdn\:\:\times\:
 \ffnk{ccccc}{q}{q^{-m},q^{-m-n},-q^{-m-n},q^{-i-1}ac,q^{i}c/a}
{q,c,-c,q^{-m-2n},q^{j-2m-2n-1}}_j
\ffnk{ccccc}{q^2}{q^{j-i}ac,q^{1+i+j}c/a}{q,q^{2j}c^2}_{m+n-j}.
 \end{eqnarray*}
\end{thm}

\begin{exam}[$\ell=0,m=1$ in Theorem \ref{thm-n}]
\begin{eqnarray*}
&&\xxqdn\xxqdn\xxqdn\qdn{_4\phi_3}\ffnk{ccccc}{q;q}{a,q/a,q^{-n},-q^{-1-n}}{c,q^{-2n}/c,-q}\\
&&\xxqdn\xxqdn\xxqdn\qdn\:\:=\:
\frac{1}{(1-q^{1+n})(1+q^nc)}\frac{(ac;q^2)_{1+n}(qc/a;q^2)_{1+n}}{(c;q)_{1+2n}}
\\&&\xxqdn\xxqdn\xxqdn\qdn\:\:-\:\:\frac{q^{1+n}}{(1-q^{1+n})(1+q^nc)}\frac{(ac/q;q^2)_{1+n}(c/a;q^2)_{1+n}}{(c;q)_{1+2n}}.
 \end{eqnarray*}
\end{exam}

\begin{exam}[$\ell=1,m=1$ in Theorem \ref{thm-n}]
\begin{eqnarray*}
&&\xqdn\qdn{_4\phi_3}\ffnk{ccccc}{q;q}{a,q^2/a,q^{-n},-q^{-1-n}}{c,q^{1-2n}/c,-q}\\
&&\xqdn\qdn\:\:=\:
\frac{q^2(1+q^{n}a)}{(q-a)(1+q^{n}c)(q+q^{n}c)(1-q^{1+n})}\frac{(ac/q^2;q^2)_{1+n}(qc/a;q^2)_{1+n}}{(c;q)_{2n}}
\\&&\xqdn\qdn\:\:-\:\:
\frac{q(a+q^{2+n})}{(q-a)(1+q^{n}c)(q+q^{n}c)(1-q^{1+n})}\frac{(ac/q;q^2)_{1+n}(c/a;q^2)_{1+n}}{(c;q)_{2n}}.
\end{eqnarray*}
\end{exam}

Setting $u=-\ell,\,v=m$ in \eqref{whipple-relation-b} and evaluating
the $_4\phi_3$-series on the right hand side by Theorem \ref{thm-b},
we obtain the following theorem.

\begin{thm}\label{thm-o}
For two complex numbers $\{a,c\}$ and two nonnegative integers
$\{\ell, m\}$ with $m\leq n$, there holds
\begin{eqnarray*}
&&\xqdn{_4\phi_3}\ffnk{ccccc}{q;q}{a,q^{1-\ell}/a,q^{-n},-q^{m-n}}
{c,q^{1-\ell+m-2n}/c,-q}=\ffnk{ccccc}{q}{a,q^{\ell-1}ac}{c,q^{\ell-1}a^2}_{\ell}
\ffnk{ccccc}{q}{q^{1-m+n},-q^{\ell}c}{q^{n-m+\ell}c,-q}_{n}
\\&&\xqdn\:\:\times\:
\sum_{i=0}^{\ell}\sum_{j=0}^{m}q^{i+(m-n)j+\binom{j+1}{2}}\frac{1}{c^i}
 \frac{1-q^{1-2\ell+2i}/a^2}{1-q^{1-2\ell}/a^2}
 \ffnk{ccccc}{q}{q^{-\ell},c/a,q^{1-2\ell}/a^2}{q,q^{2-2\ell}/ac,q^{2-\ell}/a^2}_i
 \\&&\xqdn\:\:\times\:
 \ffnk{ccccc}{q}{q^{-m},q^{-n},-q^{-n},q^{2\ell-i-1}ac,q^{i}c/a}
{q,q^{\ell}c,-q^{\ell}c,q^{m-2n},q^{j-2n-1}}_j
\ffnk{ccccc}{q^2}{q^{2\ell-i+j}ac,q^{1+i+j}c/a}{q,q^{2\ell+2j}c^2}_{n-j}.
 \end{eqnarray*}
\end{thm}

\begin{exam}[$\ell=1,m=0$ in Theorem \ref{thm-o}]
\begin{eqnarray*}
&&\xxqdn\xxqdn\xqdn{_4\phi_3}\ffnk{ccccc}{q;q}{a,1/a,q^{-n},-q^{-n}}{c,q^{-2n}/c,-q}\\
&&\xxqdn\xxqdn\xqdn\:\:=\:
\frac{1-ac}{1+a}\frac{(q^2ac;q^2)_{n}(qc/a;q^2)_{n}}{(c;q)_{1+2n}}
\\&&\xxqdn\xxqdn\xqdn\:\:+\:\:
\frac{a-c}{1+a}\frac{(qac;q^2)_{n}(q^2c/a;q^2)_{n}}{(c;q)_{1+2n}}.
  \end{eqnarray*}
\end{exam}

\begin{exam}[$\ell=1,m=1$ in Theorem \ref{thm-o}]
\begin{eqnarray*}
&&\xqdn\xxqdn{_4\phi_3}\ffnk{ccccc}{q;q}{a,1/a,q^{-n},-q^{1-n}}{c,q^{1-2n}/c,-q}\\
&&\xqdn\xxqdn\:\:=\:
\frac{1+aq^n}{(1+a)(1+q^n)}\frac{(ac;q^2)_{n}(qc/a;q^2)_{n}}{(c;q)_{2n}}
\\&&\xqdn\xxqdn\:\:+\:\:
\frac{a+q^n}{(1+a)(1+q^n)}\frac{(qac;q^2)_{n}(c/a;q^2)_{n}}{(c;q)_{2n}}.
 \end{eqnarray*}
\end{exam}

Taking $u=-\ell,\,v=-m$ in \eqref{whipple-relation-b} and
calculating the $_4\phi_3$-series on the right hand side by Theorem
\ref{thm-c}, we get the following theorem.

\begin{thm}\label{thm-p}
For two complex numbers $\{a,c\}$ and two nonnegative integers
$\{\ell, m\}$, there holds
\begin{eqnarray*}
&&\qdn\xqdn{_4\phi_3}\ffnk{ccccc}{q;q}{a,q^{1-\ell}/a,q^{-n},-q^{-m-n}}
{c,q^{1-\ell-m-2n}/c,-q}=\ffnk{ccccc}{q}{a,q^{\ell-1}ac}{c,q^{\ell-1}a^2}_{\ell}
\ffnk{ccccc}{q}{q^{1+m+n},-q^{\ell}c}{q^{\ell+m+n}c,-q}_{n}
\\&&\qdn\xqdn\:\:\times\:
\sum_{i=0}^{\ell}\sum_{j=0}^{m}q^{i-nj+\binom{j+1}{2}}\frac{(-1)^j}{c^i}
 \frac{1-q^{1-2\ell+2i}/a^2}{1-q^{1-2\ell}/a^2}
 \ffnk{ccccc}{q}{q^{-\ell},c/a,q^{1-2\ell}/a^2}{q,q^{2-2\ell}/ac,q^{2-\ell}/a^2}_i
 \\&&\qdn\xqdn\:\:\times\:
 \ffnk{ccccc}{q}{q^{-m},q^{-m-n},-q^{-m-n},q^{2\ell-i-1}ac,q^{i}c/a}
{q,q^{\ell}c,-q^{\ell}c,q^{-m-2n},q^{j-2m-2n-1}}_j
\ffnk{ccccc}{q^2}{q^{2\ell-i+j}ac,q^{1+i+j}c/a}{q,q^{2\ell+2j}c^2}_{m+n-j}.
 \end{eqnarray*}
\end{thm}

\begin{exam}[$\ell=1,m=1$ in Theorem \ref{thm-p}]
\begin{eqnarray*}
&&\xqdn{_4\phi_3}\ffnk{ccccc}{q;q}{a,1/a,q^{-n},-q^{-1-n}}{c,q^{-1-2n}/c,-q}\\
&&\xqdn\:\:=\:
\frac{1-aq^{1+n}}{(1+a)(1-q^{1+n})}\frac{(ac;q^2)_{1+n}(qc/a;q^2)_{1+n}}{(c;q)_{2+2n}}
\\&&\xqdn\:\:+\:\:
\frac{a-q^{1+n}}{(1+a)(1-q^{1+n})}\frac{(qac;q^2)_{1+n}(c/a;q^2)_{1+n}}{(c;q)_{2+2n}}.
 \end{eqnarray*}
\end{exam}

With the change of the parameters $\ell$ and $m$, Theorems
\ref{thm-a} and \ref{thm-b}-\ref{thm-p} can produce more concrete
formulas. Due to the limit of space, the corresponding results will
not be laid out in the paper.

\acknowledgements

The authors are grateful to the reviewers for helpful comments.


\nocite{*}
\bibliographystyle{plainnat}
\bibliography{q-Watson}

\end{document}